\documentclass[12pt]{article}
\usepackage[english]{babel}
\usepackage{epsfig}
\usepackage{amsfonts}
\usepackage{amssymb}
\usepackage{amsmath}
\usepackage{amsthm}
\usepackage{latexsym}
\usepackage{graphicx}
\usepackage{bm}
\usepackage{tabularx}
\usepackage{booktabs}
\usepackage{enumerate}
\usepackage{caption}
\usepackage{wrapfig}
\usepackage{mathtools}
\usepackage{eqnarray}
\usepackage{enumitem}
\usepackage[titletoc]{appendix}
\usepackage[numbers]{natbib}
\usepackage{comment}
\usepackage{bbm}
\usepackage{url}
\usepackage{subfig}
\usepackage{hyperref}
\usepackage{color}
\usepackage{float}
\usepackage{tikz}
\usetikzlibrary{arrows.meta, positioning, shapes.geometric, calc}
\usepackage{hyperref}
\hypersetup{colorlinks=true, citecolor=blue, linkcolor=red, urlcolor=blue}
\usepackage{bbm}
\usepackage{url}
\usepackage{color}
\usepackage{xcolor}
\usepackage[section]{placeins}
\usepackage[table]{xcolor}
\usepackage{booktabs}
\usepackage{subcaption}
\newcommand{\E}{\mathbb{E}}
\newcommand{\Var}{\mathrm{Var}}
\theoremstyle{plain}
\newtheorem{theorem}{Theorem}[section]

\theoremstyle{definition}

\theoremstyle{remark}
\newtheorem{remark}{Remark}[section]

\newcommand{\Cov}{\operatorname{Cov}}

\title{Queues with Rechargeable Servers}
\author{ Eliezer Fuentes-Quezada and Jamol Pender \footnote{Corresponding Author} \\ School of Operations Research and Information Engineering \\ Cornell
University \\ {eif8@cornell.edu, jjp274@cornell.edu}
}
\date{\today}

\begin{document}
\maketitle

\begin{abstract}
We introduce an Erlang--S$^{*}$ queue with stochastic server unavailability motivated by charging dynamics in drone delivery systems. Servers enter a charging state after service with probability $p$ and return at rate $\gamma$, while customers may abandon. We develop fluid and diffusion approximations for the joint process $(Q,S)$. In strictly underloaded and overloaded regimes, the diffusion limits reduce to Ornstein–Uhlenbeck processes, enabling closed-form moment approximations and tractable staffing rules. At the critical boundary, however, the drift becomes non-smooth, and the limiting diffusion transitions into a continuous regime-switching process between two operational phases. This shift alters the covariance structure and calls for a new staffing rule driven by diffusion-scale fluctuations of the gap $Q - S$. To address this challenge, we introduce a new Gaussian closure method, which remains tractable even in the non-smooth boundary regime. Numerical experiments confirm the accuracy of the approximations and the resulting staffing prescriptions across overloaded, underloaded and critical regimes.

\end{abstract}

\section{Introduction}

Drone delivery, an emerging logistics innovation, relies on unmanned aerial vehicles (UAVs) to transport goods directly to consumers, bypassing traditional road networks and leveraging the sky, \citet{grippa2019drone, cokyasar2021designing, pinto2020network, huang2020drone, shavarani2019congested, seakhoa2019revenue, khamidehi2022dynamic, munishkin2023traffic, ghosh2023performance}. From a queueing theory perspective, drone delivery introduces a new class of service mechanisms.  Unlike traditional delivery methods such as vans and trucks that follow fixed routes and perform batch deliveries, drones may serve requests individually and adaptively, resembling a multi-server queueing system with spatially distributed arrivals.  This shift requires rethinking classical queueing models to include factors such as drone battery life (leading to server vacations), weather dependent service rates, and real-time demand clustering.  Understanding the implications of these factors is crucial for designing efficient drone delivery systems. 

Drone delivery is particularly important for urgent deliveries such as medical supplies to remote or hard to traverse places.  For example, Zipline in Rwanda, operates a drone delivery service that helps deliver vital medical supplies like blood and vaccines to remote and hard to reach medical facilities.  The drone delivery improves access to healthcare and ensuring critical supplies reach patients much faster than traditional delivery methods, see for example \cite{choi2017optimization, chowdhury2017drones, boutilier2022drone, liu2022energy, chen2024courier, lejeune2025drone}.  Moreover, it has reduced maternal mortality by ensuring timely access to healthcare supplies, \citet{zailani2020drone}. Throughout the paper we treat the Erlang-$S^*$ model as a stylized queueing model motivated by drone delivery, rather than a complete operational model for any specific system; the assumptions of identical servers, exponential charging times, and state-independent charging probabilities are standard first-model idealizations, and extensions to heterogeneous fleets and spatial structure are deferred to future work.

In addition, drone delivery offers environmental benefits by reducing carbon emissions compared to traditional delivery vehicles.  Drones are typically electric-powered, which can lead to lower greenhouse gas emissions, especially when charged using renewable energy sources, \citet{goodchild2018delivery, figliozzi2020carbon, brown2024last}.  This makes drone delivery an attractive option for companies looking to reduce their environmental footprint while also improving delivery speed and efficiency.
However, drone delivery systems face unique challenges that impact their queueing dynamics.  One of the most significant challenges is the limited battery life of drones, which necessitates periodic recharging.  This introduces a stochastic element to server availability, as drones (servers) may leave the system to recharge after completing deliveries, leading to fluctuations in service capacity.  Additionally, factors such as weather conditions, air traffic regulations, and payload limitations further complicate the service process.  These challenges require new queueing models that can accurately capture the dynamics of drone delivery systems and inform effective staffing and scheduling strategies.
\begin{table}[h!] 
\centering
\begin{tabular}{|@{}l|c|c|c|@{}|}
 \hline
\textbf{Drone Model}       & \textbf{Charging Time} & \textbf{Flight Time} & \textbf{Max Payload} \\  \hline
DJI Mavic Air 2            & 1 hr 40 min            & 34 min              & 0.2 kg              \\ \hline
DJI Matrice 300 RTK        & 1.5 hrs (per battery)  & 55 min              & 2.7 kg              \\ \hline
Autel EVO II Pro           & 1.5 hrs                & 40 min              & 0.9 kg              \\ \hline
Parrot Anafi USA           & 2 hrs                  & 32 min              & 0.5 kg              \\ \hline
Skydio X10                   & 1.5 hrs                & 40 min              & 0.4 kg              \\  \hline
\end{tabular}
\caption{Drone Models Specifications}
\end{table} \label{drone_times}

Table \ref{drone_times} summarizes the operational characteristics of several commercial drone models, highlighting their charging times, flight durations, and payload capacities, \citet{DJI_2025, Autel_2025, Parrot_2025, Skydio_2025}. These specifications directly influence the performance of a drone-based service system modeled as a queue, where drones act as servers that periodically leave to recharge. For instance, the DJI Matrice 300 RTK, with a flight time of 55 minutes and a payload of 2.7 kg, represents a high-capacity but energy-demanding server, while the DJI Mavic Air 2 and Skydio X10, with shorter flight times and lighter payloads, correspond to lower-capacity servers that must recharge more frequently. These heterogeneities translate into variability in effective service rates and server availability, key determinants of system congestion and customer waiting times.

In the context of the queueing model analyzed in this paper, the data in Table \ref{drone_times} illustrate how drone design and energy constraints shape the stochastic dynamics of the service process. Longer charging durations increase the fraction of time drones are unavailable, effectively reducing the system’s active service capacity. Conversely, models with shorter recharging cycles or longer flight durations contribute to smoother operations and lower abandonment probabilities. By capturing such trade-offs, Table \ref{drone_times} provides a realistic basis for parameterizing the model’s service and recharge rates, enabling the analysis of fluid and diffusion limits that inform energy-aware staffing.

On the other hand, with respect to the customer experience, capturing and managing wait time for deliveries is crucial.  Customers expect timely deliveries, and any delays can lead to dissatisfaction and loss of business, see for example \citet{massey2013gaussian, van2013spectral, pender2014gram, pender2014poisson, knessl2015transient, pender2015truncated, massey2018dynamic,  legros2024transient}.  Therefore, understanding and optimizing the queueing dynamics in drone delivery systems is essential for maintaining high service levels and customer satisfaction. Unfortunately, there are very few papers that study drone delivery with a queueing emphasis.  In particular, we find it important to include abandonment and charging times of drones (server vacations).  Research by \citet{dimou2011single, dudin2024queueing} explores queueing models with abandonment and server vacations, however, these papers either consider discrete time queueing models or single server queueing models.  The paper by \citet{altman2006analysis} does consider multiple servers, however, the model assumes that all idle servers go for vacations, which is not practical for this drone delivery setting. 

Finally, \citet{azriel2019erlang} studies queueing systems with stochastic server availability through an exponential vacation model motivated by call-center applications. In their framework, servers enter vacations after service completion with a fixed probability, a modeling assumption that is shared with the present paper. However, server availability in their model is restored through multiple mechanisms: servers may return from vacation spontaneously at an exponential rate, or become available immediately with a certain probability when customers are waiting in the queue. In contrast, in our model servers return from charging exclusively at an exponential rate, independent of the queue length. This modeling choice, while simpler, leads to fundamentally different system dynamics: server availability is coupled to service completions but decoupled from instantaneous queue congestion, reflecting the operational realities of drone recharging. As a result, the induced queue–server dependence, as well as the resulting fluid and diffusion limits, differ substantially from those that arise in vacation-based models with queue--dependent server return mechanisms.

The vacation-queue literature provides important methodological context for the present work. The classical foundations; the decomposition results of \citet{fuhrmann1985stochastic}, the survey of \citet{doshi1986queueing}, and the monographs of \citet{takagi1991} and \citet{tian2006vacation}, establish the steady-state theory for single-server and synchronous multi-server vacation models, typically without customer abandonment. The state-independent Bernoulli post-service vacation mechanism in the Erlang--S$^*$ model traces its lineage to the single-server Bernoulli-schedule vacation queue of \citet{keilson1986oscillating}; the present paper is the first to extend this mechanism to the many-server regime with abandonment and derive the resulting fluid and diffusion limits. The multi-server setting with impatient customers was studied by \citet{altman2006analysis}, who obtain closed-form steady-state results for M/M/c under synchronous pool vacations; the key distinction from our model is that vacations in their framework are triggered by server idleness and all servers vacation simultaneously, whereas in the Erlang--S$^*$ model vacations are individual and are triggered probabilistically upon each service completion. More recent work by \citet{yue2014analysis}, \citet{goswami2016phase}, and \citet{bouchentouf2020multiserver} extends the vacation-with-abandonment analysis to richer arrival processes and working-vacation mechanisms, but all of these papers  operate via exact matrix-analytic or generating-function methods and do not derive fluid or diffusion limits, do not work in any many-server scaling regime, and do not obtain the joint queue--server covariance $\Cov(Q,S)$. The absence of many-server asymptotic theory for vacation queues with abandonment is not incidental: the standard diffusion-limit machinery for multi-server queues requires the fluid equilibrium to lie in the interior of a smooth face of the drift, a condition that fails precisely at the critical boundary $\{q^*=s^*\}$ that arises in our model. This is the technical gap the present paper fills.

In the many-server diffusion literature, the foundational results are \citet{halfin1981heavy} for the Erlang--C queue and \citet{garnett2002designing} for the Erlang--A queue with abandonment, both of which establish one-dimensional Ornstein--Uhlenbeck (OU) limits in the QED (quality- and efficiency-driven) regime. Extensions to general patience-time distributions appear in \citet{zeltyn2005call}, and staffing rules for simultaneous delay and abandonment targets are developed in \citet{mandelbaum2009staffing}. All of these results treat the server count as a deterministic constant, so the queue--server covariance is identically zero by construction. The one-dimensional piecewise-linear diffusion theory of \citet{browne1995piecewise} underlies both the Halfin--Whitt and Erlang--A limits and is systematically developed for multidimensional piecewise-OU processes in \citet{dieker2013positive}; the Erlang--S$^*$ boundary limit is the first two-dimensional instance of this class arising from a coupled queue--server system with endogenous server availability. The effects of intermittently unavailable servers on many-server queues were studied by \citet{pang2009service} and \citet{pang2009heavy}, who obtain diffusion limits for queues with exogenous service interruptions (an environmental on/off process that modulates the entire server pool) in the QED and ED regimes. In those models the interruption process is exogenous and the limit is one-dimensional; in the Erlang--S$^*$ model the charging process is endogenous, triggered by each individual server's service completion, and the limit is genuinely two-dimensional with a nonzero $\Cov(Q,S)$. The most closely related paper technically is \citet{aveklouris2019bounds}, who derive a two-dimensional reflected OU limit with piecewise-linear drift for an electric-vehicle charging model; their limit is the nearest predecessor of the bivariate piecewise-OU structure we obtain. However, their model has no customer abandonment, the charging mechanism is a processor-sharing capacity constraint rather than a post-service Bernoulli decision, and the closed-form covariance between the queue and the server process is not characterized. At the boundary both drift branches are simultaneously active and no single Jacobian exists, so the interior Lyapunov calculation does not apply. We do not establish the corresponding limit theorem here; instead we quantify the departure from the one-sided approximations by means of a standardized distance $z(c)$ (Section~\ref{sec:boundary-diffusion}) and close the exact moment equations under a bivariate Gaussian estimation (Section~\ref{sec:gauss-moment-near-boundary}), obtaining approximations for $\operatorname{Var}(S)$ and $\operatorname{Cov}(Q,S)$ that agree with exact stationary computation to within a few percent across the boundary. The identification of the limiting piecewise-linear diffusion, and of its stationary law, is left for future work. We work throughout in the many-server fluid and diffusion scaling rather than the QED heavy-traffic regime; the QED-regime analysis of the Erlang--S$^*$ model, where the boundary between OU and switching-diffusion behavior would interact with the $\beta$ parameter of the QED scaling, is an open problem we leave for future work. The present paper is, to the best of our knowledge, the first to derive joint fluid and diffusion limits for a multi-server queue with state-independent post-service Bernoulli vacations, customer abandonment, and an endogenous stochastic server-availability process, yielding closed-form $\Cov(Q,S)$ and covariance-aware staffing rules. 

Concurrent work by \citet{asgari2026solving} studies the same Markovian queue under the correspondence \(p=1-\alpha\), \(\gamma=\xi\), and \(\theta=\gamma_{\mathrm A}\). They formulate the stationary system as an infinite-level, level-dependent QBD and use controlled truncation to compute finite-system performance measures. Related work in \citet{asgari2025queueing} derives a one-dimensional Halfin--Whitt approximation. Our approach instead retains the joint state \((Q,S)\), yielding explicit approximations for its stationary means, variances, and covariance, together with a treatment of the transition near \(Q=S\). These formulas make the effects of the model parameters transparent and avoid solving the full stationary matrix system, although they do not provide the finite-system error guarantees of the matrix-analytic approach.

Thus, in this paper, we introduce a new queueing framework that captures a fundamental feature of drone delivery systems: stochastic server availability driven by post-service charging requirements. We model drones as servers that probabilistically enter a charging state after completing service and return to availability at an exponential rate, while customers may abandon if their waiting time exceeds their patience. Within this framework, we characterize system behavior using fluid and diffusion limit theorems, deriving explicit expressions for the steady-state means, variances, and covariances of the joint queue length and server availability process. These results enable tractable staffing policies that achieve prescribed targets for delay and abandonment probabilities.

While motivated by drone delivery, the modeling framework and analytical insights developed in this paper extend naturally to a broad class of service systems with intermittently unavailable servers. In healthcare settings, for example, medical staff, diagnostic equipment, or operating rooms often require recovery, cleaning, or setup periods following service completion, during which they are temporarily unavailable. Similarly, in robotic and automated service systems, robots must periodically recharge or undergo maintenance after completing tasks, leading to endogenous fluctuations in service capacity. Related dynamics arise in shared mobility and micromobility systems, such as e-bike or scooter fleets (see for example \citet{sun2017optimal,  tao2020impact, pender2020stochasticscooter, tao2021stochastic, xie2024empowering, tao2020limit,  pla2025demand, liu2025shared}), where vehicles become unavailable while charging or undergoing battery swaps.

More broadly, the Erlang–S$^{*}$ framework captures service environments in which server availability is coupled to workload through post-service recovery, rather than determined solely by idleness or external scheduling. By explicitly accounting for this coupling and the resulting queue/server dependence, the proposed model provides a foundation for energy-aware and reliability-driven staffing decisions in modern service systems. The analytical tractability of the fluid and diffusion approximations makes the framework particularly suitable for large-scale systems where capacity planning must balance congestion, abandonment, and resource recovery constraints. 

\subsection{Contributions of the Paper}
In this section, we outline the major contributions of the paper.  
\begin{enumerate}
    \item A stochastic model (Erlang--S$^{*}$) where servers take i.i.d.~stochastic ``vacations'' to charge.
    \item Fluid limits with steady-state closed forms that highlight the systematic capacity drain from charging.
    \item Diffusion (OU) limits providing regime-specific formulas for $\Var (Q)$, $\Var (S)$, $\Cov (Q,S)$ and their approximation on the regime boundaries.
    \item Staffing rules for meeting delay and abandonment targets that explicitly incorporate covariance.
    \item We validate our results with stochastic simulation.
\end{enumerate}

\subsection{Organization of the Paper} 
In this section, we provide an outline of how this paper is organized. Section~\ref{sec:model} defines the stochastic model, its dynamics and model primitives. Section~\ref{sec:fluid} presents the fluid limit and analyzes the steady-state regimes of the fluid limit. Section~\ref{sec:diffusion} derives diffusion limits and derives new insights about the structure of the covariance. Section~\ref{sec:staffing} develops staffing schedules (deterministic and bivariate-normal) to stabilize delay probabilities. Section~\ref{sec:abandon} develops staffing schedules to stabilize abandonment probabilities. Finally, Section~\ref{sec:numerics} reports numerical experiments and all non-trivial proof details are provided in the Appendix~\ref{app:proofs}.

\section{The Erlang--S\texorpdfstring{$^*$}{*} model}
\label{sec:model}

\subsection{Model construction}
\label{subsec:model-construction}

We extend the Erlang--A queue by allowing servers to become temporarily
unavailable after service. Customers arrive according to a Poisson process
with rate $\lambda$, service times are exponential with rate $\mu$, and each
waiting customer abandons at rate $\theta$. The system contains $c$ identical
servers. After completing service, a server enters a charging state with
probability $p$ and returns after an exponential time with rate $\gamma$.
This structure models service systems with reusable resources, such as delivery
drones subject to battery constraints.

Let $Q(t)$ denote the number of customers in the system and $S(t)$ the number
of active servers, so that $c-S(t)$ servers are charging. The process
$(Q(t),S(t))$ has state space
\[
\mathbb Z_{\geq0}\times\{0,1,\ldots,c\}.
\]
At state $(q,s)$, service completions occur at rate
$\mu(q\wedge s)$, abandonments at rate $\theta(q-s)^+$, and charging
completions at rate $\gamma(c-s)$. Each service completion decreases the
number of active servers with probability $p$.

\begin{figure}[ht]
    \centering
    \includegraphics[width=0.8\linewidth]{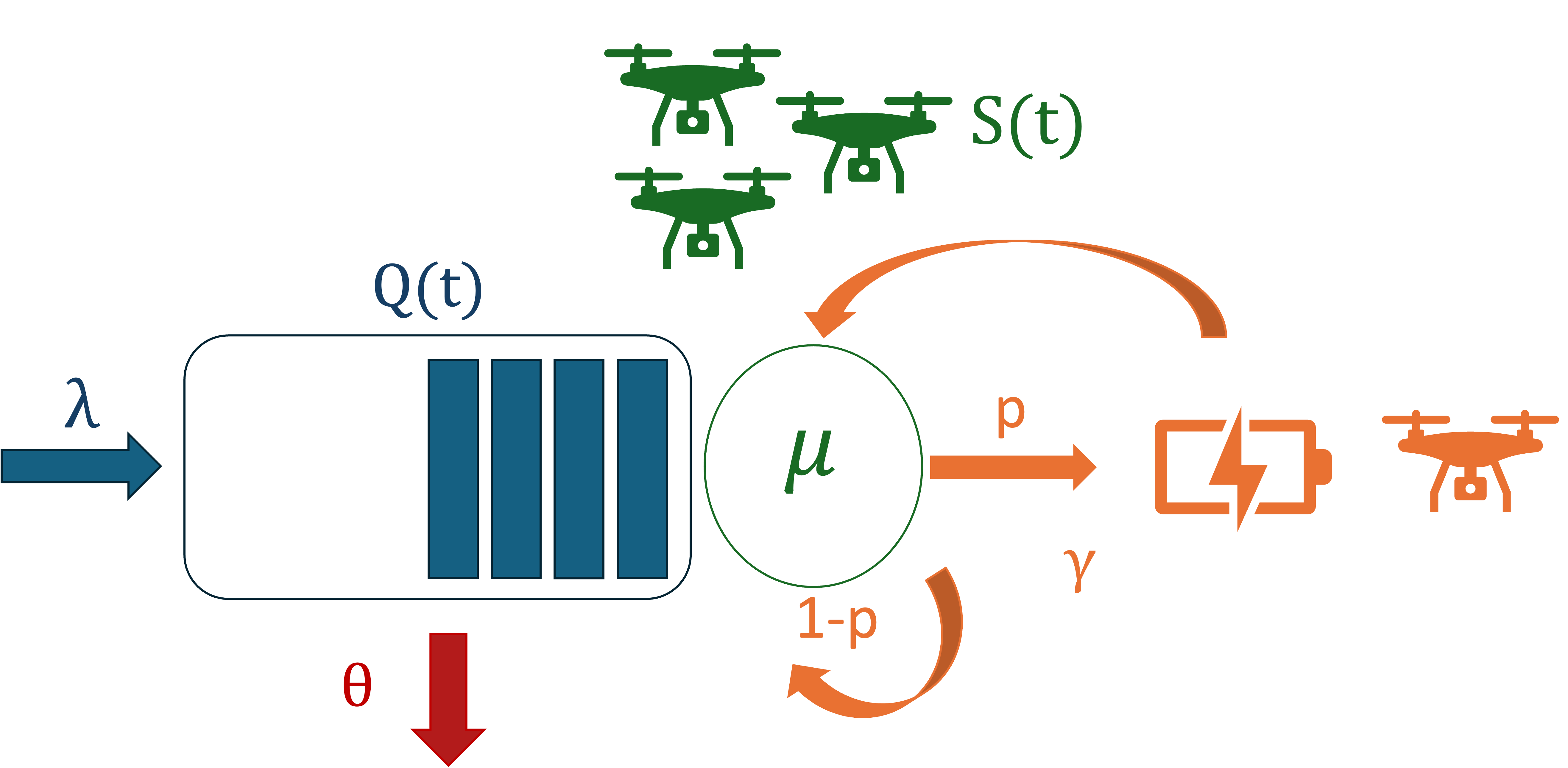}
    \caption{Schematic representation of the Erlang--$S^*$ model.}
    \label{fig:placeholder}
\end{figure}

The model is related to the Erlang--S framework of
\citet{azriel2019erlang}, which also models stochastic server availability.
The distinguishing feature here is that charging begins probabilistically
after service and ends at rate $\gamma$, independently of queue congestion.
This removes state-dependent availability controls while retaining the
interaction between demand and available capacity.

The model assumes homogeneous servers, Poisson arrivals, exponential service,
abandonment, and charging times, and no spatial constraints. These assumptions
yield a tractable Markovian model; extensions are discussed in
Section~\ref{sec:conclusions}.

The Markovian primitives considered here coincide, after a change of notation, with the stochastic-agent-unavailability model of \citet{asgari2026solving}. Whereas that work uses a level-dependent QBD formulation to compute finite-system stationary probabilities, we retain the joint queue-capacity state in order to study its fluid limit, covariance structure, and behavior near the switching boundary.

\subsection{Stochastic Dynamics}
Let $\{ N_A(t),N_B(t),N_C(t),N_{Ab}(t) \}_{t\geq 0}$ be independent unit-rate Poisson processes for arrivals, service completions, charge completions, and abandonments, respectively. Let $\{\xi_i\}$ be i.i.d.~Bernoulli($p$) indicating whether the completing server goes to charge, with $ P(\xi_i = 1) = p, \quad P(\xi_i = 0) = 1-p,$ indicating whether the server goes to charging after the $i^{th}$ service completion. Then
\begin{align}
Q(t) &= Q(0) + N_A(\lambda t) - N_B\!\left(\mu\!\int_0^t \min\{Q(s),S(s)\}\,ds\right) \nonumber\\ 
&\quad - N_{Ab}\!\left(\theta\!\int_0^t (Q(s)-S(s))^+ ds\right),
\label{eq:Q-int}\\
S(t) &= S(0) + N_C\!\left(\gamma\!\int_0^t (c-S(s))\,ds\right) - \sum_{i\le N_B\left(\mu\int_0^t \min\{Q(s),S(s)\}\,ds\right)} \xi_i.
\label{eq:S-int}
\end{align}
We can interpret the dynamics as follows: service completions remove customers; a fraction $p$ of completions send the server to charge; charging returns occur at rate $\gamma$ per charging server; abandonments occur at rate $\theta$ per waiting customer. The server process is autonomous but not independent of the queue due to service completions.

It is important to note that, in this model, unlike the classical Erlang-A system with fixed capacity, stochastic server availability creates feedback between congestion and capacity, so that variability and covariance directly affect system performance. In the following two sections, we study these second-order effects in detail.



\subsection{Forward Equations}\label{sec:forward-equations}
At state \((Q,S)\), the possible transitions are
\[
\begin{array}{ccl}
(Q,S)\longrightarrow(Q+1,S)
&\quad\text{at rate}\quad&
\lambda,
\\[1mm]
(Q,S)\longrightarrow(Q-1,S)
&\quad\text{at rate}\quad&
(1-p)\mu(Q\wedge S),
\\[1mm]
(Q,S)\longrightarrow(Q-1,S-1)
&\quad\text{at rate}\quad&
p\mu(Q\wedge S),
\\[1mm]
(Q,S)\longrightarrow(Q-1,S)
&\quad\text{at rate}\quad&
\theta(Q-S)^+,
\\[1mm]
(Q,S)\longrightarrow(Q,S+1)
&\quad\text{at rate}\quad&
\gamma(c-S).
\end{array}
\]

The infinitesimal generator of the Markov chain is therefore
\begin{align}
\mathcal{L}f(Q,S)
={}&
\lambda\bigl[f(Q+1,S)-f(Q,S)\bigr]
\nonumber\\
&+(1-p)\mu(Q\wedge S)
   \bigl[f(Q-1,S)-f(Q,S)\bigr]
\nonumber\\
&+p\mu(Q\wedge S)
   \bigl[f(Q-1,S-1)-f(Q,S)\bigr]
\nonumber\\
&+\theta(Q-S)^+
   \bigl[f(Q-1,S)-f(Q,S)\bigr]
\nonumber\\
&+\gamma(c-S)
   \bigl[f(Q,S+1)-f(Q,S)\bigr].
\label{eq:generator}
\end{align}

For any sufficiently integrable function \(f\), Dynkin's formula gives
\[
\frac{d}{dt}\E[f(Q(t),S(t))]
=
\E\!\left[
\mathcal{L}f(Q(t),S(t))
\right].
\]

Appendices~\ref{app:exact-moment-equations} and \ref{app:gaussian-moment-closure} contain the computations of the different moments from these forward equations and its Gaussian closure, which we will use in Section~\ref{sec:boundary-diffusion} and \ref{sec:boundary-staffing-rule} as a tool for the approximation near the boundary. 
\section{Fluid Approximation}
\label{sec:fluid}
We analyze the Erlang–S$^{*}$ system under fluid scaling. Let $\bar Q^{n}(t)=Q^{n}(t)/n$ and $\bar S^{n}(t)=S^{n}(t)/n$ denote the scaled queue length and server processes. Their dynamics are described by the following equations.
\begin{align}
Q^n(t) &= Q^n(0) + N_A(\lambda n t) - N_B\left( \mu \int_0^t \min(Q^n(s), S^n(s)) ds \right) \nonumber\\
&\quad - N_{Ab}\left( \theta \int_0^t (Q^n(s) - S^n(s))^+ ds \right), \\
S^n(t) &= S^n(0) + N_C\left( \gamma \int_0^t (c - S^n(s)/n) ds \right)
- \sum_{i=1}^{N_B\left( \mu \int_0^t \min(Q^n(s), S^n(s)) ds \right)} \xi_i.
\end{align}
With initial conditions $\bar Q^n(0) \Rightarrow q(0)$, $\bar S^n(0) \Rightarrow s(0)$, we have the following fluid limit theorem.
\begin{theorem}[Fluid Limit for the Erlang--$S^*$ Queue]\label{thm:fluid-limit}
Consider the scaled stochastic processes
\[
\bar Q^n(t)=\frac{Q^n(t)}{n},\qquad
\bar S^n(t)=\frac{S^n(t)}{n}.
\]
Then for every fixed $T>0$,
\[
\sup_{0\le t\le T}
\big\|(\bar Q^n(t),\bar S^n(t))
       -( q(t), s(t))\big\|
\;\xrightarrow{P}\;0,
\]
where $(q(t), s(t))$ is the unique solution on $[0,T]$ of
\begin{align}
\frac{dq(t)}{dt} &= \lambda - \mu\,\min\{q(t),s(t)\} - \theta\,(q(t)-s(t))^+, \label{eq:fluid-q}\\
\frac{ds(t)}{dt} &= \gamma\,(c - s(t)) - p\mu\,\min\{q(t),s(t)\}. \label{eq:fluid-s}
\end{align}
\end{theorem}

\begin{proof}
	The proof can be found in Appendix \ref{app:fluid-proof}.
\end{proof}

Under standard functional LLNs for Poisson processes and random sums, $(\bar Q^n(t),\bar S^n(t))$ converges to the unique solution $(q(t),s(t))$ of \eqref{eq:fluid-q} and \eqref{eq:fluid-s}, which we will analyze next.

\subsection{Steady-State Regimes}
In this section, we analyze the steady state limit of the fluid limit we obtained in the previous section.

\begin{eqnarray} \label{fluid_limits_1}
\begin{split}
    \frac{d q(t)}{dt} &=& \lambda - \mu ( q(t) \wedge s(t) ) - \theta ( q(t) - s(t) )^+ \\
    \frac{d s(t)}{dt} &=& \gamma ( c - s(t)) - p \mu ( q(t) \wedge s(t) ).
\end{split}
\end{eqnarray}

\begin{theorem} \label{thm:steady_state_limits}
    Let $(q^*,s^*)$ be the steady state limits of Equations \ref{fluid_limits_1}, if $ \frac{\lambda}{\mu} + \frac{\lambda p}{\gamma} \leq c $, then
    \begin{eqnarray}
        q^* &=& \frac{\lambda}{\mu} \\
        s^* &=& c - \frac{\lambda p}{\gamma} 
    \end{eqnarray}
    and if $ \frac{\lambda}{\mu} + \frac{\lambda p}{\gamma} > c $, then we have
        \begin{eqnarray}
        q^* &=& \frac{\lambda - \frac{\gamma \mu c}{\gamma + p \mu } }{\theta} + \frac{\gamma c}{\gamma + p \mu }  \\
        s^* &=& \frac{\gamma c}{\gamma + p \mu }.
    \end{eqnarray}
\end{theorem}
\begin{proof}
    We will divide the proof into underloaded (UL) and overloaded (OL) cases.  The first case considers the situation where there is no abandonment and the second case is where there is abandonment.  In the case with no abandonment in steady state, we have $q^* \leq s^*$.  Using this, we have in the differential equations set to zero that
    \begin{eqnarray} \label{fluid_limits}
    0 &=& \lambda - \mu q^*  \\
    0 &=& \gamma ( c - s^* ) - p \mu q^*.
	\end{eqnarray}
	This proves the first case and since $q^* \leq s^*$, this is also the same condition in the statement of the result.  
	The second case assumes that there is abandonment and therefore $q^* \geq s^*$.  This implies
	\begin{eqnarray} \label{fluid_limits_2}
    	0 &=& \lambda - \mu s^* - \theta ( q^* - s^* ) \\
    	0 &=& \gamma ( c - s^*) - p \mu  s^* .
	\end{eqnarray}
\end{proof}
\begin{remark}
In the underloaded regime, the steady-state effect of stochastic charging
manifests as an additive reduction in the effective number of available
servers, given by $\lambda p/\gamma$. In contrast, in the overloaded regime,
charging induces a multiplicative scaling of server availability by the factor
$\gamma/(\gamma+p\mu)$, reflecting the long-run fraction of time that servers
are available.
\end{remark}


\section{Diffusion limits and covariance structure}
\label{sec:diffusion}

The Erlang-$S^*$ model induces a nontrivial and operationally significant covariance between the queue length and the number of available servers. The fluid drift is piecewise affine, with switching boundary $q=s$. Away from this boundary, the diffusion limit is obtained by linearizing the drift on the active face. At a fixed interior equilibrium, this produces a two-dimensional Ornstein--Uhlenbeck process whose stationary covariance captures the dependence between the queue length and the number of available servers. The boundary case $q^*=s^*$ requires a separate approximation and is treated in Sections~\ref{sec:boundary-diffusion} and~\ref{sec:gauss-moment-near-boundary}.

Let
\[
\bar X^n(t)
:=
\begin{pmatrix}
\bar Q^n(t)\\
\bar S^n(t)
\end{pmatrix},
\qquad
x(t)
:=
\begin{pmatrix}
q(t)\\
s(t)
\end{pmatrix},
\]
where $x(t)$ is the fluid limit from Section~\ref{sec:fluid}, and define
\[
\widehat X^n(t):=\sqrt n\bigl(\bar X^n(t)-x(t)\bigr).
\]

\begin{theorem}[Diffusion limit away from the switching boundary]
\label{thm:diffusion_limit}
Fix $T>0$ and suppose that the fluid path remains on one face:
\[
\inf_{0\leq t\leq T}|q(t)-s(t)|>0.
\]
If
\[
\widehat X^n(0)\Rightarrow\widehat X(0),
\]
where $\widehat X(0)$ is square-integrable, then
\[
\widehat X^n\Rightarrow\widehat X
\qquad\text{in }D([0,T],\mathbb R^2),
\]
where $\widehat X$ is the unique strong solution of
\begin{equation}
d\widehat X(t)
=
J_r\widehat X(t)\,dt
+
\Sigma_r^{1/2}(t)\,dW(t).
\label{eq:interior-diffusion-limit}
\end{equation}
Here $r\in\{\mathrm{UL},\mathrm{OL}\}$ denotes the active regime, $J_r$ is
the Jacobian of the fluid drift on that face, and $\Sigma_r(t)$ is the
primitive-noise covariance-rate matrix evaluated along the fluid path.
\end{theorem}

\begin{proof}
See Appendix~\ref{app:diffusion-proof}.
\end{proof}

When the fluid is initialized at an interior equilibrium $x^*=(q^*,s^*)$,
the coefficients in \eqref{eq:interior-diffusion-limit} are constant and the
limit is an Ornstein--Uhlenbeck process. The eigenvalues of $J_{\mathrm{UL}}$
are $-\mu$ and $-\gamma$, while those of $J_{\mathrm{OL}}$ are $-\theta$
and $-(\gamma+p\mu)$. Hence both interior equilibria are locally
exponentially stable, and the stationary covariance matrix $V_r$ is the
unique solution of
\begin{equation}
J_rV_r+V_rJ_r^\top+\Sigma_r=0.
\label{eq:stationary-lyapunov}
\end{equation}
The corresponding covariance formulas are computed below.

\subsection{Underloaded Regime ($q^*<s^*$)}\label{subsec:underloaded-regime}

\begin{theorem}\label{thm:diffusion_UL}
In the underloaded regime the Jacobian and diffusion matrices are
\[
  J_{\mathrm{UL}}=\begin{pmatrix}-\mu&0\\-p\mu&-\gamma\end{pmatrix},
  \qquad
  \Sigma_{\mathrm{UL}}=\begin{pmatrix}2\lambda&p\lambda\\p\lambda&2p\lambda\end{pmatrix},
\]
and the stationary second moments are
\[
  \Var(Q)\approx\frac{\lambda}{\mu},\qquad
  \Var(S)\approx\frac{\lambda p}{\gamma},\qquad
  \Cov(Q,S)\approx 0.
\]
\end{theorem}

\begin{proof}
The proof is in Appendix~\ref{app:underloaded-proof}.
\end{proof}

\subsection{Overloaded Regime ($q^*>s^*$)}\label{subsec:overloaded-regime}

\begin{theorem}\label{thm:diffusion_OL}
In the overloaded regime the Jacobian and diffusion matrices are
\[
  J_{\mathrm{OL}}=\begin{pmatrix}-\theta&\theta-\mu\\0&-(\gamma+p\mu)\end{pmatrix},
  \qquad
  \Sigma_{\mathrm{OL}}=\begin{pmatrix}2\lambda&p\mu s^*\\p\mu s^*&2p\mu s^*\end{pmatrix},
\]
and the stationary second moments are
\begin{align*}
  \Var(Q)  &\approx\frac{\lambda}{\theta} + \frac{\theta - \mu}{\theta}\frac{c\,\gamma p\mu}{(\gamma+p\mu)^2}
  \cdot\frac{\gamma+\theta+p\mu-\mu}{\theta+\gamma+p\mu},\qquad
  \Var(S)  \approx\frac{c\,\gamma p\mu}{(\gamma+p\mu)^2},\\[4pt]
  \Cov(Q,S)&\approx
  \frac{c\,\gamma p\mu}{(\gamma+p\mu)^2}
  \cdot\frac{\gamma+\theta+p\mu-\mu}{\theta+\gamma+p\mu}.
\end{align*}
\end{theorem}

\begin{proof}
The proof is in Appendix~\ref{app:overloaded-proof}.
\end{proof}


\subsection{A standardized boundary region}
\label{sec:boundary-diffusion}

\begin{figure}[ht]
    \centering
    \includegraphics[width=0.5\linewidth]{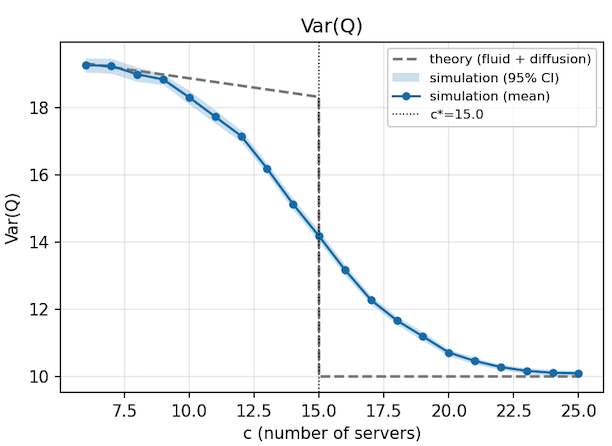}~\includegraphics[width=0.5\linewidth]{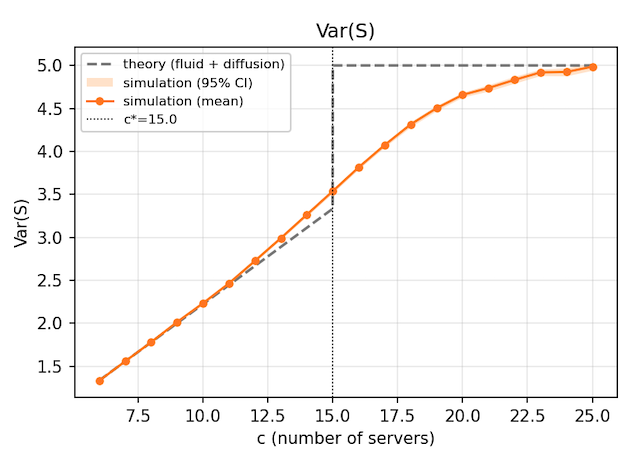}
    \caption{Variance of $Q$ and $S$ gap between fluid and diffusion theory and simulation. Parameters $\lambda =10,\, \mu = 1,\, \theta = 0.5,\, p=0.5,$ and $\gamma = 1$.}
    \label{fig:variances_sim_vs_theory}
\end{figure}

Read as functions of the staffing level, the two approximations of Sections~\ref{subsec:underloaded-regime} and~\ref{subsec:overloaded-regime} disagree with each other at the critical capacity $c^*=\lambda/\mu+\lambda p/\gamma$. For the parameters of Figure~\ref{fig:variances_sim_vs_theory}, where $c^*=15$, the overloaded formulas give $\Var(Q)=18.3$, $\Var(S)=3.3$ and $\Cov(Q,S)=1.7$, while the underloaded formulas give $10.0$, $5.0$ and $0$: the predicted variance of $Q$ falls by $45\%$ and the covariance vanishes, both discontinuously, as $c$ crosses $c^*$. The simulated moments, approximately $(14.2,3.6,0.4)$, vary smoothly across the boundary. Thus, the discontinuity is caused by linearizing the drift on only one of the two faces $q<s$ and $q>s$; near $q=s$, the process visits both faces and neither one-sided approximation captures the full local dynamics.

To identify this transition region, define
\[
d^*(c):=s^*(c)-q^*(c),
\qquad
\sigma^2_D(c):=\Var(S-Q)
       =v_{qq}(c)+v_{ss}(c)-2v_{qs}(c),
\]
where $\sigma^2_D(c)$ is computed from the corresponding one-sided diffusion
approximation. With
\[
\kappa=\frac{\gamma}{\gamma+p\mu},
\]
the equilibrium gap is
\begin{equation}
d^*(c)=
\begin{cases}
\dfrac{\mu\kappa}{\theta}(c-c^*), & c<c^*,\\[2mm]
c-c^*, & c\geq c^*.
\end{cases}
\label{eq:fluid-equilibrium-gap}
\end{equation}
The standardized distance to the switching boundary is therefore
\begin{equation}
z(c):=\frac{d^*(c)}{\sqrt{\sigma^2_D(c)}}.
\label{eq:standardized-boundary-distance}
\end{equation}
Under a normal approximation for $S-Q$,
\[
\Pr(Q<S)\approx\Phi(z(c)),
\qquad
\Pr(Q\geq S)\approx\Phi(-z(c)).
\]

We consequently define the two-standard-deviation boundary region by
\begin{equation}
\mathcal{B}_2:=\{c:|z(c)|\leq2\}.
\label{eq:two-sd-boundary-region}
\end{equation}
Outside this region, the system occupies one face with approximate probability
at least $\Phi(2)\approx0.977$, motivating the use of the underloaded
approximation when $z(c)>2$ and the overloaded approximation when $z(c)<-2$.
Inside $\mathcal{B}_2$, switching between the two faces is too frequent to
ignore, so we use the Gaussian moment closure described next. The threshold
$2$ is an operational criterion rather than an asymptotic result.

\subsection{Gaussian moment closure near the boundary}\label{sec:gauss-moment-near-boundary}
Inside $\mathcal{B}_2$, we close the exact moment equations from
Appendix~\ref{app:gaussian-moment-closure} by approximating
\begin{equation}
\begin{pmatrix}Q\\S\end{pmatrix}
\approx
\mathcal{N}\left(
\begin{pmatrix}m_Q\\m_S\end{pmatrix},
\begin{pmatrix}
v_{qq} & v_{qs}\\
v_{qs} & v_{ss}
\end{pmatrix}
\right).
\label{eq:bivariate-normal-boundary-closure}
\end{equation}
Let
\[
m_D:=m_Q-m_S,
\qquad
\sigma_D^2:=v_{qq}+v_{ss}-2v_{qs},
\qquad
a:=\frac{m_D}{\sigma_D}.
\]
Since $D:=Q-S\approx\mathcal{N}(m_D,\sigma_D^2)$,
\begin{align}
\E[(Q-S)^+]
&\approx \sigma_D\phi(a)+m_D\Phi(a),
\label{eq:normal-positive-part-boundary}\\
\E[Q\wedge S]
&\approx m_Q-\sigma_D\phi(a)-m_D\Phi(a).
\label{eq:normal-minimum-boundary}
\end{align}
Together with the Gaussian covariance identities stated in
Appendix~\ref{app:gaussian-moment-closure}, these expressions produce a closed
nonlinear system for
\[
m_Q,\qquad m_S,\qquad v_{qq},\qquad v_{ss},\qquad v_{qs}.
\]
Its stationary solution accounts for both faces through the congestion weight
$\Phi(a)$ and provides a continuous transition between the two regime-specific
approximations.

For comparison, a simpler heuristic directly interpolates between the
one-sided diffusion formulas. For any second-order quantity $M$, define
\begin{equation}
M_{\mathrm{int}}(c)
=
\bigl(1-\Phi(z(c))\bigr)M_{\mathrm{OL}}(c)
+
\Phi(z(c))M_{\mathrm{UL}}(c).
\label{eq:gaussian-regime-mixture}
\end{equation}
This interpolation approaches the overloaded formula as $z(c)\to-\infty$ and
the underloaded formula as $z(c)\to+\infty$, but unlike the moment closure, it
is not derived from the system's moment equations.

Figure~\ref{fig:variances_vs_gauss_clos} compares both boundary corrections
with simulation.

\begin{figure}[ht]
    \centering
    \includegraphics[width=0.49\linewidth]
    {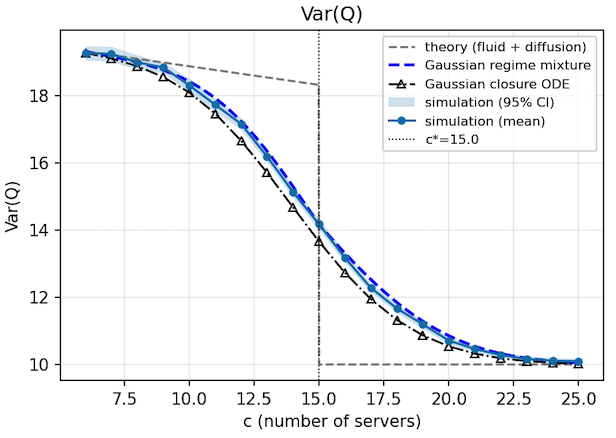}
    \hfill
    \includegraphics[width=0.49\linewidth]
    {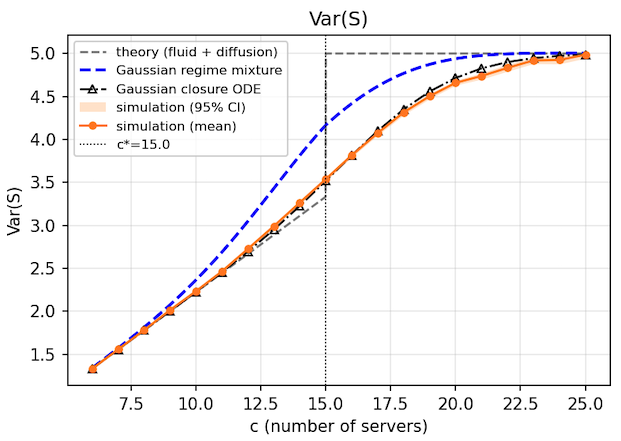}
    \caption{Simulated variances of $Q$ and $S$ compared with the Gaussian
    moment closure and the weighted interpolation. Parameters:
    $\lambda=10$, $\mu=1$, $\theta=0.5$, $p=0.5$, and $\gamma=1$.}
    \label{fig:variances_vs_gauss_clos}
\end{figure}


\section{Staffing for Delay Targets}
\label{sec:staffing}
A central performance measure in our system is the probability of delay
\begin{eqnarray}
    \mathbb{P} \left( Q \geq S  \right),
\end{eqnarray}
the probability that an arrival must wait before service begins. In the classical Erlang-S model the server capacity is a fixed constant. This performance measure is valuable for making staffing decisions and making sure that service level agreements (SLA's) are met. However, in our setting, the number of available servers, $S$, fluctuates because servers without sufficient charge are forced to charge up. This section develops two staffing approximations: one assuming a deterministic approximation for $S$ and one incorporating the stochastic fluctuations of $S$.

\subsection{Deterministic Approximation for Number of Servers}\label{sec:delay-determ}

Suppose $Q$ is approximately normal, $Q \sim N(q^*, q^*)$, and replace the random available-server count by its fluid steady-state limit i.e. $S = s^*$. And let $\Phi(x)$ and $\varphi(x)$ be the cdf and pdf of a normal distribution. Then 

\begin{eqnarray}
    \mathbb{P} \left( Q \geq s^*  \right) =  \overline{\Phi} \left( \frac{s^* -q^*}{\sqrt{q^*}} \right) = \epsilon.
\end{eqnarray}

\begin{theorem}[Deterministic-server staffing]\label{thm:determ_server_staff-delay}
    Let $q^*$ and $s^*$ denote the steady-state fluid limits of the queue length and available-server process, respectively. If the delay probability target is $\epsilon \in (0,1)$ and $Q \sim N(q^*,q^*)$, $S=s^\ast$, then the staffing level that achieves 
    \begin{eqnarray}
        \mathbb{P} \left( Q \geq S  \right) = \epsilon.
    \end{eqnarray}
    satisfies 
    \begin{eqnarray}\label{eq:delay-prob}
        s^* = q^* + \sqrt{q^*} \overline{\Phi}^{-1} \left( \epsilon \right).
    \end{eqnarray}
    In the under-loaded regime the fluid relations $q^* =\lambda/\mu$ and $s^* = c-\lambda p /\gamma$ yield
    \begin{eqnarray}
        c_\epsilon = \frac{\lambda p}{\gamma} + \frac{\lambda}{\mu} + \sqrt{\frac{\lambda}{\mu}} \overline{\Phi}^{-1} \left( \epsilon \right).
    \end{eqnarray}
    In the overloaded regime, where $q^* = \frac{\lambda - \frac{\gamma \mu c}{\gamma + p \mu } }{\theta} + \frac{\gamma c}{\gamma + p \mu }$ and $s^* = \frac{\gamma c}{\gamma + p \mu }$, the same approximation yields 
    \begin{equation}
    \label{eq:staff-over}
    c_\epsilon \;=\;
    \frac{\gamma + p \mu}{2 \gamma \mu^{2}}
    \left[
    2 \lambda \mu \;-\; \theta(\mu - \theta)\,z_\epsilon^{2}
    \;+\;
    \sqrt{\Big(2 \lambda \mu - \theta(\mu - \theta)\,z_\epsilon^{2}\Big)^{2}
    - 4 \lambda \mu^2 (\lambda - \theta\, z_\epsilon^{2})}
    \right].
    \end{equation}
    where $z_\epsilon = \overline{\Phi}^{-1} \left( \epsilon \right).$

\end{theorem}
\begin{proof}
	The proof can be found in Appendix \ref{app:deterministic-staffing-proof}.
\end{proof}

\subsection{Normal approximation for the number of servers}
\label{sec:delay-normal}

The deterministic approximation of Section~\ref{sec:delay-determ} replaces the
number of available servers by its fluid mean $s^*$, thereby ignoring fluctuations
in $S$. Because delay occurs when $Q\geq S$, the relevant variability is that of the
gap $D:=Q-S$, not of $Q$ alone. We therefore use the regime-specific bivariate-normal
approximation
\begin{equation}
\begin{pmatrix}Q\\S\end{pmatrix}
\dot{\sim}
\mathcal N\left(
\begin{pmatrix}q^*\\s^*\end{pmatrix},
\begin{pmatrix}
v_{qq} & v_{qs}\\
v_{qs} & v_{ss}
\end{pmatrix}
\right),
\label{eq:bvn}
\end{equation}
where the means and covariances are given by
Theorems~\ref{thm:steady_state_limits},
\ref{thm:diffusion_UL}, and~\ref{thm:diffusion_OL}. It follows directly that
\begin{equation}
\Pr(Q\geq S)
\approx
\overline{\Phi}\left(
\frac{s^*-q^*}
{\sqrt{v_{qq}+v_{ss}-2v_{qs}}}
\right).
\label{eq:joint-normal-delay}
\end{equation}

Let
\[
c^*=\frac{\lambda}{\mu}+\frac{\lambda p}{\gamma},
\qquad
z_\epsilon:=\overline{\Phi}^{-1}(\epsilon).
\]
Solving \eqref{eq:joint-normal-delay} gives the following staffing rules.

\begin{theorem}[Joint-normal staffing approximation]
\label{thm:joint-normal_server_staff-delay}
Fix $\epsilon\in(0,1)$ and use the regime-specific diffusion covariances.

\medskip
\noindent\textbf{(a) Underloaded regime ($\epsilon<1/2$).}
Since
\[
s^*-q^*=c-c^*,
\qquad
\Var(Q-S)=\frac{\lambda}{\mu}+\frac{\lambda p}{\gamma}=c^*,
\]
the required staffing level is
\begin{equation}
c_\epsilon
=
c^*+z_\epsilon\sqrt{c^*}
=
\frac{\lambda}{\mu}+\frac{\lambda p}{\gamma}
+
z_\epsilon
\sqrt{\frac{\lambda}{\mu}+\frac{\lambda p}{\gamma}}.
\label{eq:c-under-cov-final}
\end{equation}
Because $z_\epsilon>0$, this solution satisfies $c_\epsilon>c^*$.

\medskip
\noindent\textbf{(b) Overloaded regime ($\epsilon>1/2$).}
Let
\[
\kappa:=\frac{\gamma}{\gamma+p\mu},
\qquad
U_A:=
\frac{\gamma p\mu^2\bigl[\mu(1-p)-\gamma\bigr]}
{\theta(\gamma+p\mu)^2(\theta+\gamma+p\mu)}.
\]
Then
\[
s^*-q^*=\frac{\mu\kappa c-\lambda}{\theta},
\qquad
\Var(Q-S)=\frac{\lambda}{\theta}+U_Ac.
\]
If $\lambda>\theta z_\epsilon^2$, the feasible staffing level is the smaller
root of
\begin{equation}
\mu^2\kappa^2c^2
-
\left(2\lambda\mu\kappa+z_\epsilon^2\theta^2U_A\right)c
+
\lambda\left(\lambda-\theta z_\epsilon^2\right)
=0.
\label{eq:quadOL}
\end{equation}
Equivalently,
\begin{equation}
c_\epsilon
=
\frac{
A_\epsilon-
\sqrt{
A_\epsilon^2
-
4\mu^2\kappa^2\lambda
\left(\lambda-\theta z_\epsilon^2\right)}
}
{2\mu^2\kappa^2},
\qquad
A_\epsilon:=
2\lambda\mu\kappa+z_\epsilon^2\theta^2U_A,
\label{eq:c-over-cov-final}
\end{equation}
and $0<c_\epsilon<c^*$.
\end{theorem}

\begin{proof}
See Appendix~\ref{app:joint-normal-correlated-proof}.
\end{proof}

The underloaded rule has the familiar square-root form, but charging affects
both the fluid requirement and the safety margin. In contrast, the
deterministic-server rule uses the smaller safety term
$z_\epsilon\sqrt{\lambda/\mu}$. The ratio of its safety margin to that in
\eqref{eq:c-under-cov-final} is
\begin{equation}
\sqrt{\frac{\lambda/\mu}
{\lambda/\mu+\lambda p/\gamma}}
=
\frac{1}{\sqrt{1+p\mu/\gamma}}.
\label{eq:detloss}
\end{equation}
Thus, ignoring server fluctuations can substantially understate the required
safety capacity. For example, when $p=0.5$, $\mu=1$, and $\gamma=0.1$, the
deterministic rule retains only $1/\sqrt{6}\approx0.41$ of the joint-normal
safety margin.

At $\epsilon=1/2$, both formulas return $c^*$. Nevertheless, the underlying
one-sided covariance approximations are discontinuous there and should not be
used close to the critical capacity. The boundary region and its Gaussian
moment-closure approximation are developed in
Section~\ref{sec:boundary-diffusion}.

\subsection{Boundary-aware staffing}
\label{sec:boundary-staffing-rule}

The staffing formulas of Section~\ref{sec:delay-normal} use one-sided diffusion
moments and may therefore be inaccurate near the critical capacity $c^*$.
This issue affects many practically relevant targets. Indeed, the underloaded
rule
\[
c_\epsilon=c^*+z_\epsilon\sqrt{c^*},
\qquad
z_\epsilon:=\overline{\Phi}^{-1}(\epsilon),
\]
lies within two standard deviations of $c^*$ whenever
\[
\overline{\Phi}(2)\leq\epsilon<\frac12.
\]
Thus, for delay targets between approximately $0.023$ and $0.5$, the staffing
decision falls precisely in the region where switching between the two drift
regimes cannot be ignored.

We therefore retain the normal delay approximation
\eqref{eq:joint-normal-delay}, but evaluate it using the Gaussian
moment-closure moments from
Section~\ref{sec:gauss-moment-near-boundary}. For $c\in\mathcal B_2$, define
\begin{equation}
\widehat z_{\mathrm B}(c)
:=
\frac{\widehat m_S(c)-\widehat m_Q(c)}
{\sqrt{
\widehat v_{qq}(c)+\widehat v_{ss}(c)-2\widehat v_{qs}(c)
}},
\label{eq:boundary-standardized-gap-staffing}
\end{equation}
where
\[
\widehat m_Q(c),\quad
\widehat m_S(c),\quad
\widehat v_{qq}(c),\quad
\widehat v_{ss}(c),\quad
\widehat v_{qs}(c)
\]
denote the stationary Gaussian-closure moments.

Let
\[
c_-:=\inf\mathcal B_2,
\qquad
c_+:=\sup\mathcal B_2,
\]
and define the boundary-aware standardized gap
\begin{equation}
z_{\mathrm{BA}}(c)
:=
\begin{cases}
z_{\mathrm{OL}}(c), & c<c_-,\\[1mm]
\widehat z_{\mathrm B}(c), & c_-\leq c\leq c_+,\\[1mm]
z_{\mathrm{UL}}(c), & c>c_+,
\end{cases}
\label{eq:boundary-aware-standardized-gap}
\end{equation}
where
\[
z_r(c):=
\frac{s^*(c)-q^*(c)}
{\sqrt{v_D^r(c)}},
\qquad
r\in\{\mathrm{OL},\mathrm{UL}\}.
\]
Assuming that $z_{\mathrm{BA}}(c)$ is nondecreasing and that the closure
matches continuously with the one-sided approximations at $c_-$ and $c_+$,
the staffing prescription is
\begin{equation}
\widehat c_\epsilon
=
\left\lceil
\inf\left\{
c\geq0:
z_{\mathrm{BA}}(c)\geq z_\epsilon
\right\}
\right\rceil.
\label{eq:complete-boundary-aware-staffing}
\end{equation}
Equivalently, $\widehat c_\epsilon$ is the smallest integer capacity satisfying
\[
\overline{\Phi}\bigl(z_{\mathrm{BA}}(c)\bigr)\leq\epsilon.
\]

Although $\mathcal B_2$ is defined by the symmetric standardized condition
$|z(c)|\leq2$, it is generally asymmetric in capacity. On the underloaded
side,
\[
c_+-c^*=2\sqrt{c^*},
\]
whereas the overloaded endpoint satisfies
\begin{equation}
\frac{\mu\kappa}{\theta}
\frac{c_--c^*}
{\sqrt{v_D^{\mathrm{OL}}(c_-)}}
=-2.
\label{eq:lower-boundary-endpoint}
\end{equation}
Hence $c^*-c_-$ need not equal $c_+-c^*$. This asymmetry arises from both
the different slopes of the fluid gap and the regime-dependent variance of
$S-Q$.

\section{Abandonment targets}
\label{sec:abandon}

Let $N_{\mathrm{Ab}}(t)$ and $N_A(t)$ denote the cumulative numbers of
abandonments and arrivals. Under stationarity and rate balance, the long-run
abandonment fraction is
\begin{equation}
\alpha
:=
\lim_{t\to\infty}\frac{N_{\mathrm{Ab}}(t)}{N_A(t)}
=
\frac{\theta}{\lambda}\E[(Q-S)^+].
\label{eq:abandonment-fraction}
\end{equation}

To approximate this quantity, let $D:=Q-S$ and use the Gaussian closure from
Section~\ref{sec:gauss-moment-near-boundary}. Define
\[
m(c):=\widehat m_Q(c)-\widehat m_S(c),
\qquad
\sigma_D^2(c)
:=
\widehat v_{qq}(c)+\widehat v_{ss}(c)-2\widehat v_{qs}(c),
\qquad
a(c):=\frac{m(c)}{\sigma_D(c)}.
\]
Since $D$ is approximated by
$\mathcal N(m(c),\sigma_D^2(c))$, the standard normal loss formula gives
\begin{equation}
\widehat\alpha(c)
=
\frac{\theta}{\lambda}
\left[
\sigma_D(c)\varphi\bigl(a(c)\bigr)
+
m(c)\Phi\bigl(a(c)\bigr)
\right].
\label{eq:alpha-of-c}
\end{equation}
The Gaussian closure can be used throughout the parameter range and avoids
switching formulas near $c^*$. Away from the boundary, its stationary moments
approach the corresponding one-sided diffusion approximations.

For a target $\varepsilon\in(0,1)$, the abandonment staffing rule is
\begin{equation}
c_\varepsilon
:=
\min\left\{
c\in\mathbb Z_+:
\widehat\alpha(c)\leq\varepsilon
\right\}.
\label{eq:aband-rule}
\end{equation}
When $\widehat\alpha(c)$ is nonincreasing, as observed in our numerical
experiments, \eqref{eq:aband-rule} can be computed by bisection. Otherwise, an
integer search should be used to select the smallest capacity satisfying the
target.

For interpretation, a closed form is available under the underloaded
diffusion approximation. Let
\[
c^*:=\frac{\lambda}{\mu}+\frac{\lambda p}{\gamma},
\qquad
u:=\frac{c-c^*}{\sqrt{c^*}},
\qquad
h(u):=\varphi(u)-u\overline{\Phi}(u).
\]
Using
\[
m(c)=-(c-c^*),
\qquad
\sigma_D^2(c)=c^*,
\]
in \eqref{eq:alpha-of-c} yields
\begin{equation}
\widehat\alpha_{\mathrm{UL}}(c)
=
\frac{\theta\sqrt{c^*}}{\lambda}\,h(u).
\label{eq:alpha-underloaded}
\end{equation}
Consequently,
\begin{equation}
c_\varepsilon^{\mathrm{UL}}
=
\left\lceil
c^*
+
\sqrt{c^*}\,
h^{-1}\left(
\frac{\lambda\varepsilon}{\theta\sqrt{c^*}}
\right)
\right\rceil.
\label{eq:aband-underloaded-rule}
\end{equation}
This expression is appropriate when the resulting standardized margin is
positive and sufficiently far from the boundary. Near $c^*$, the staffing
level should instead be computed from the full closure through
\eqref{eq:aband-rule}.

Additional properties of $(Q-S)^+$ and $(S-Q)^+$ are provided in
Appendix~\ref{app:positive-part}.


\section{Numerical Experiments}
\label{sec:numerics}

This section validates the fluid and diffusion approximations developed in Sections~\ref{sec:fluid}--\ref{sec:diffusion} by means of detailed event--driven simulations.  Our goals are threefold: (i) to examine the accuracy of the fluid limits for $(Q(t),S(t))$ across a wide range of parameters; (ii) to assess the quality of the second--order (diffusion) moments $v_{qq},v_{ss},v_{qs}$; and (iii) to evaluate the operational staffing rules for delay and abandonment targets presented in Sections~\ref{sec:staffing} and~\ref{sec:abandon}.

\begin{figure}[ht]
    \centering
    \includegraphics[width=0.75\linewidth]{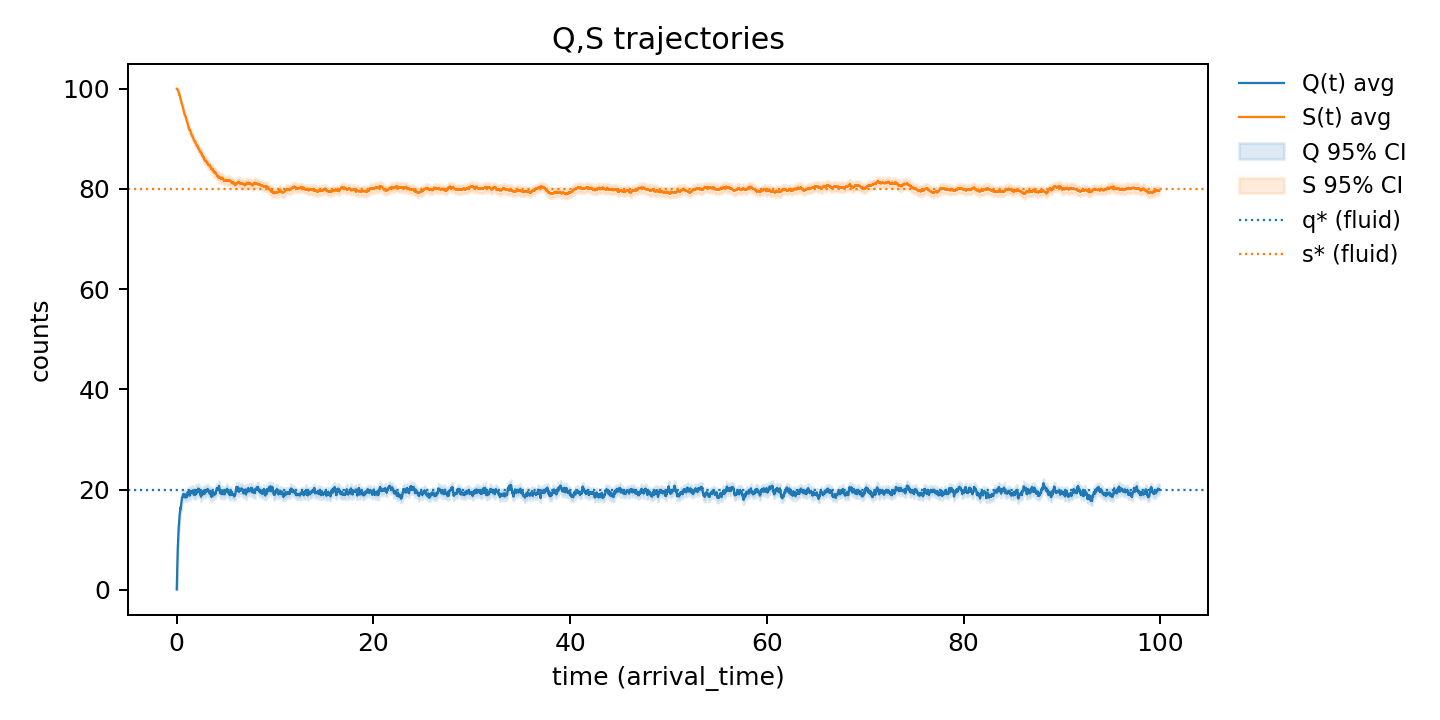}
    \caption{$Q(t)$ and $S(t)$, $\lambda = 100,\, \mu = 5 , \, \theta = 1,\, p=0.1,\, \gamma = 0.5$ and $c=100$ (UL)}
    \label{fig:QS-trajectories_UL}
\end{figure}

\begin{figure}[t]
    \centering
    \includegraphics[width=0.75\linewidth]{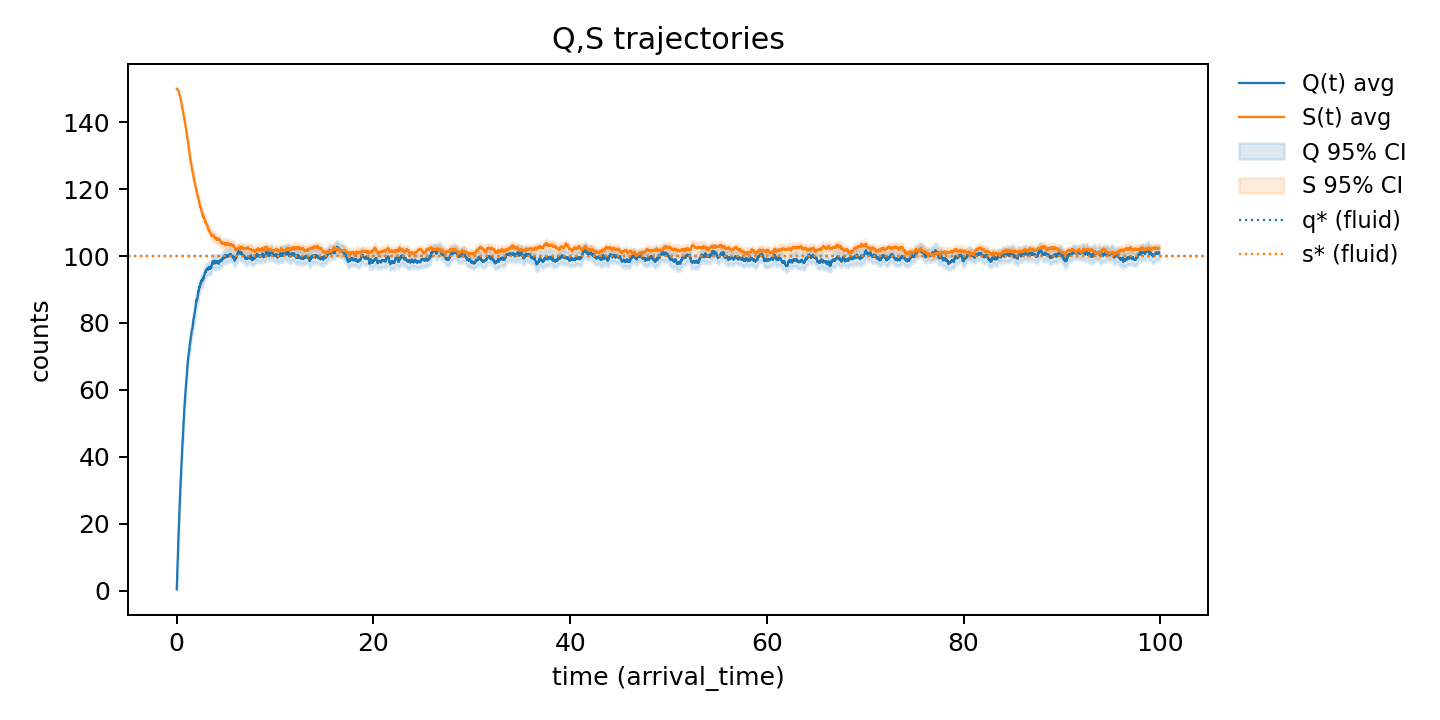}
    \caption{$Q(t)$ and $S(t)$, $\lambda = 100,\, \mu = 1 , \, \theta = 1,\, p=0.5,\, \gamma = 1$ and $c=100$ (OL)}
    \label{fig:QS-trajectories_OL}
\end{figure}

\begin{figure}[t]
    \centering
    \includegraphics[width=0.75\linewidth]{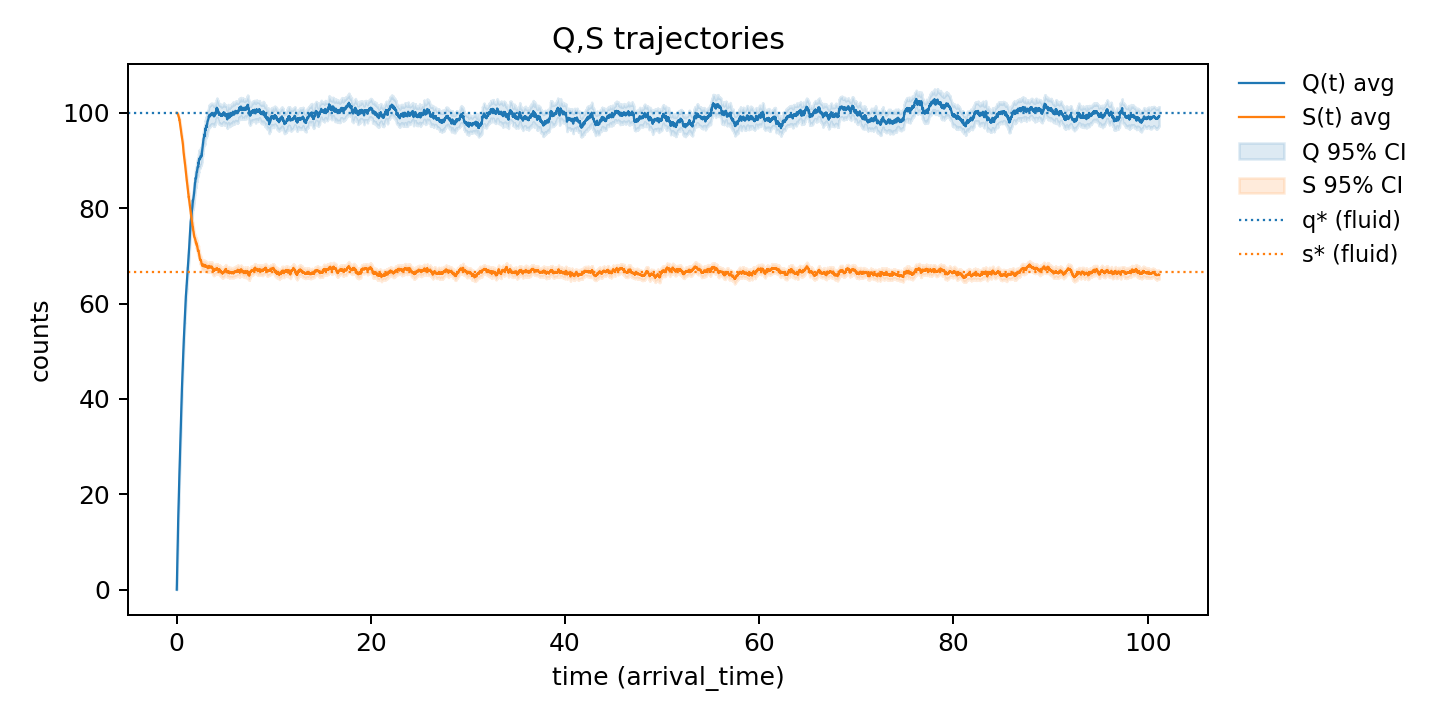}
    \caption{$Q(t)$ and $S(t)$, $\lambda = 100,\, \mu = 1 , \, \theta = 1,\, p=0.5,\, \gamma = 1,\, c=150$ (boundary)}
    \label{fig:QS-trajectories_boundary}
\end{figure}






\begin{figure}[t]
    \centering
    \includegraphics[width=0.85\linewidth]{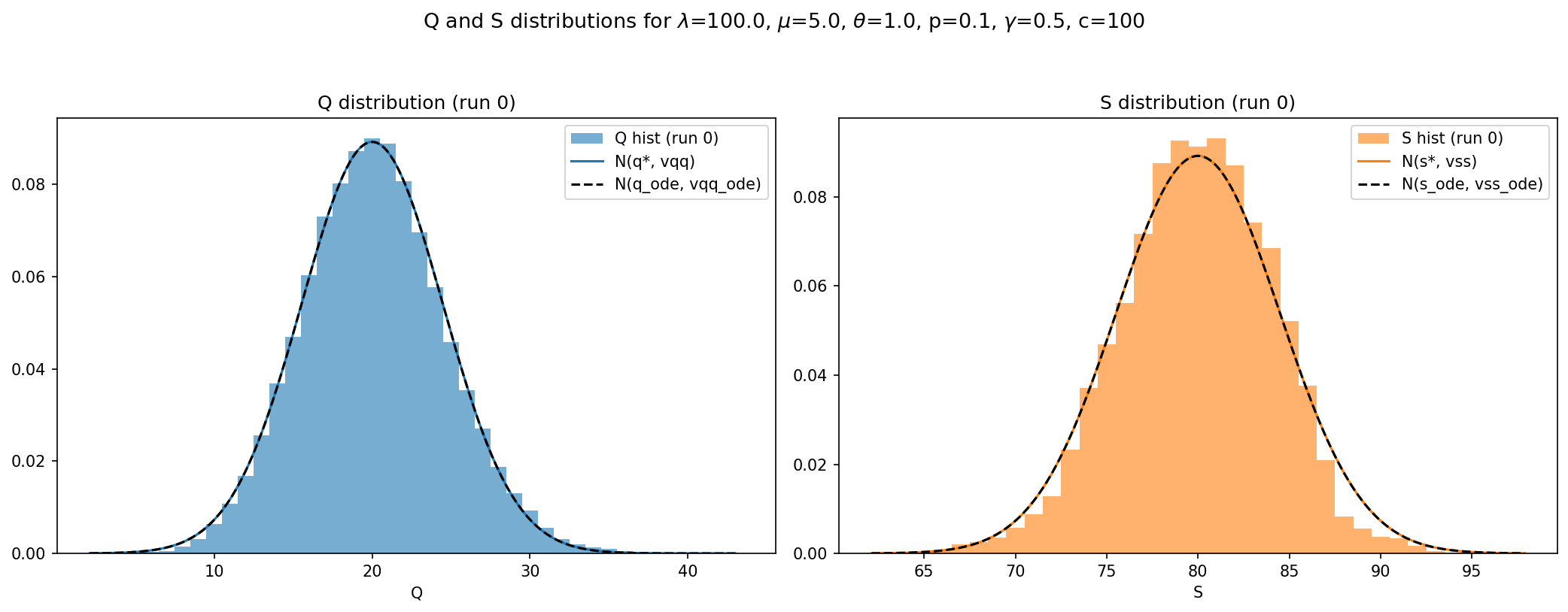}
    \caption{Distribution of $Q$ and $S$ for UL parameters of Fig.~\ref{fig:QS-trajectories_UL}}
    \label{fig:dist_q_s_UL}
\end{figure}

\begin{figure}[t]
    \centering
    \includegraphics[width=0.85\linewidth]{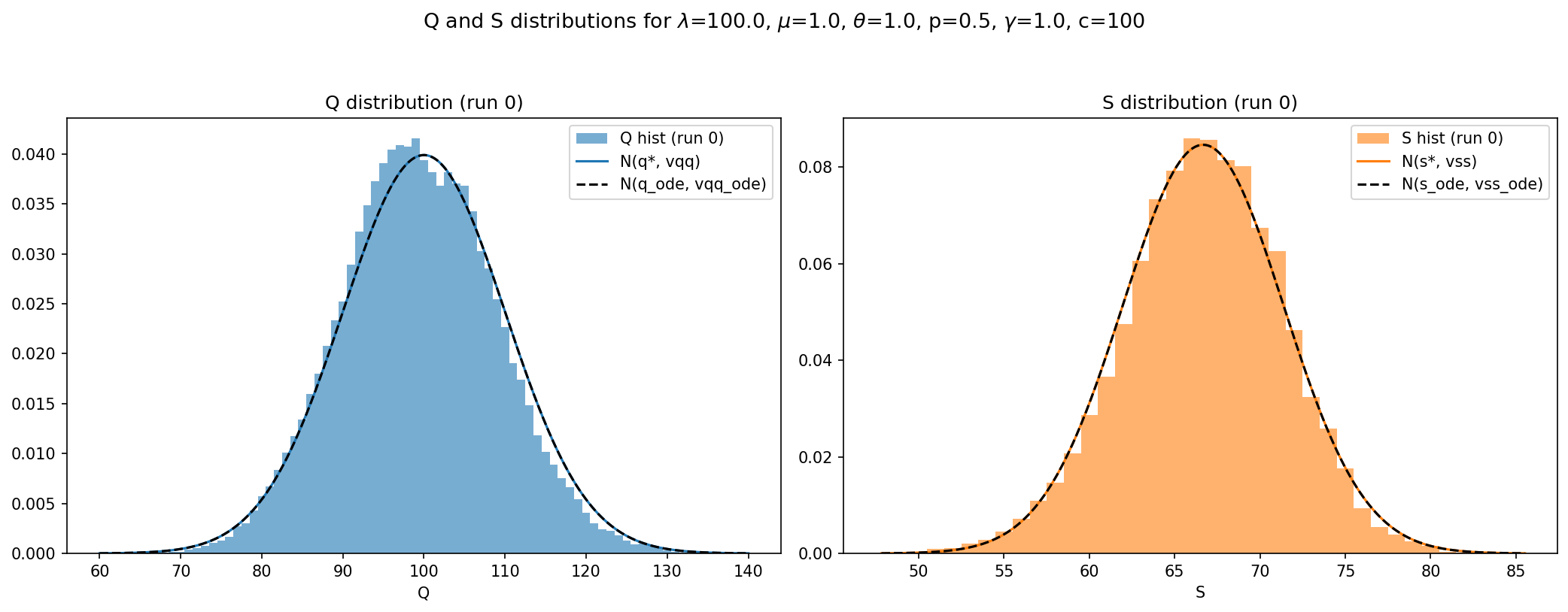}
    \caption{Distribution of $Q$ and $S$ for OL parameters of Fig.~\ref{fig:QS-trajectories_OL}}
    \label{fig:dist_q_s_OL}
\end{figure}

\begin{figure}[t]
    \centering
    \includegraphics[width=0.85\linewidth]{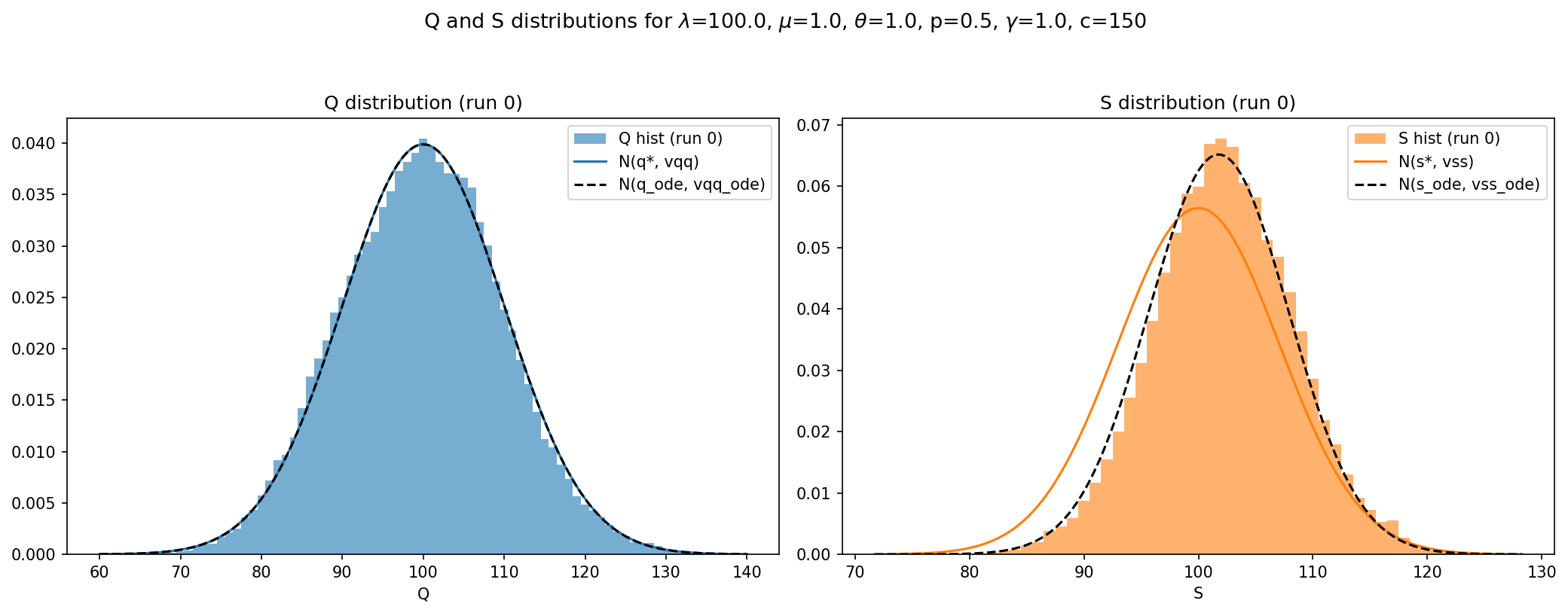}
    \caption{Distribution of $Q$ and $S$ for boundary parameters of Fig.~\ref{fig:QS-trajectories_boundary}}
    \label{fig:dist_q_s_boundary}
\end{figure}





\subsection{Simulation methodology}

We implement an event--driven simulation that tracks the system state $(Q(t),S(t))$ through arrival events, service completions, abandonment events, and transitions between active and charging server states.  The simulator records a per--arrival snapshot of $(Q,S)$, time--averaged performance statistics, and cumulative abandonment and delay counts. Since the model is naturally event driven rather than time driven, we post process the event sequence into a regularly sampled time series to enable rolling moment estimates and trajectory plots.

To ensure robustness, we conduct a systematic parameter sweep covering over 5995 configurations (with $100{,}000$ customers each); covering a broad range of arrival rates, service rates, abandonment rates, return rates, vacation (charging) probabilities, and staffing level that were selected from the following ranges:
\[
\lambda \in \{80,\,100,\,120\}, \qquad
\mu \in \{0.5,\,1,\,10\}, \qquad
\theta \in \{0.5,\,1,\,2\},
\]
\[
p \in \{0.1,\,0.5,\,0.8\}, \qquad
\gamma \in \{0.1,\,0.5,\,1,\,10\}, \qquad c\in [1,1000].
\]
The staffing levels used in this sweep were chosen so we could show the transition from UL to OL. For each configuration, we ran a single long simulation to estimate steady-state performance measures (delay probability, abandonment rate, average queue, etc.).

Because each run contained many customers, Monte Carlo noise for these coarse metrics is relatively small compared to the variation across parameter settings. We therefore use this sweep primarily to study qualitative trends and to compare fluid/diffusion approximations across a broad range of parameters, rather than to make precise point estimates with confidence intervals. For visualization, we focus on three representative configurations corresponding to UL, OL, and near-critical regimes.

For visualization and interpretation, we also include a collection of representative runs that illustrate the qualitative behavior of the dynamics in both UL and OL regimes. For these, we simulate $100$ independent runs with $30{,}000$ customers per run. We then average trajectories and moments across runs and construct pointwise confidence bands (which turn out to be very narrow).

In short, for the small number of representative parameter sets, we ran 100 independent replications and constructed confidence intervals using the classical-i.i.d. approach. For the large parameter sweep (5995 configurations) we used single long runs with time averaging to estimate steady-state quantities; these are primarily used to study trends and to compare fluid/diffusion approximations, rather than to report high-precision point estimates for each configuration.
\begin{figure}[t]
    \centering
    \includegraphics[width=0.75\linewidth]{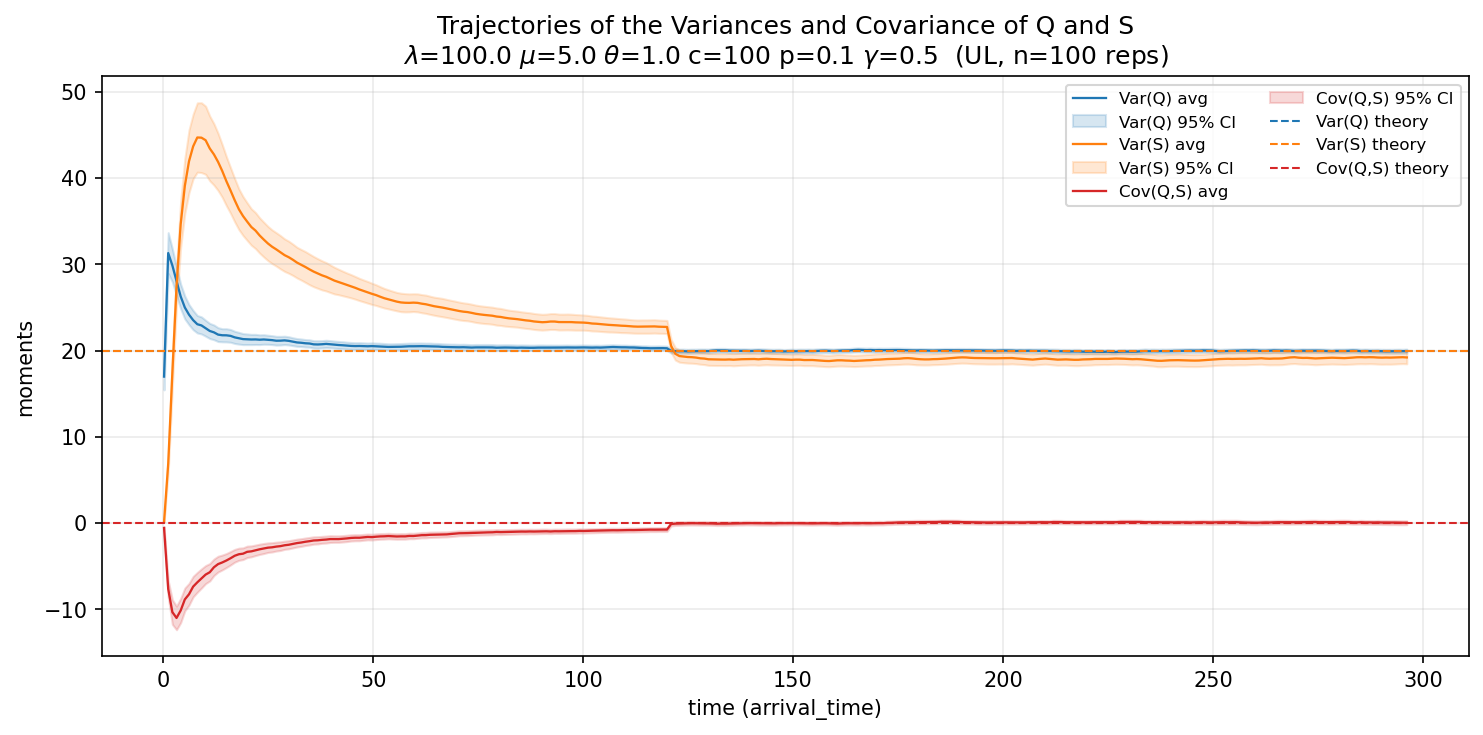}
    \caption{Variance and covariance with $\lambda = 100,\, \mu = 5 , \, \theta = 1,\, p=0.1,\, \gamma = 0.5$ and $c=100$ (UL)}
    \label{fig:var_cov_traj_UL}
\end{figure}

\begin{figure}[t]
    \centering
    \includegraphics[width=0.75\linewidth]{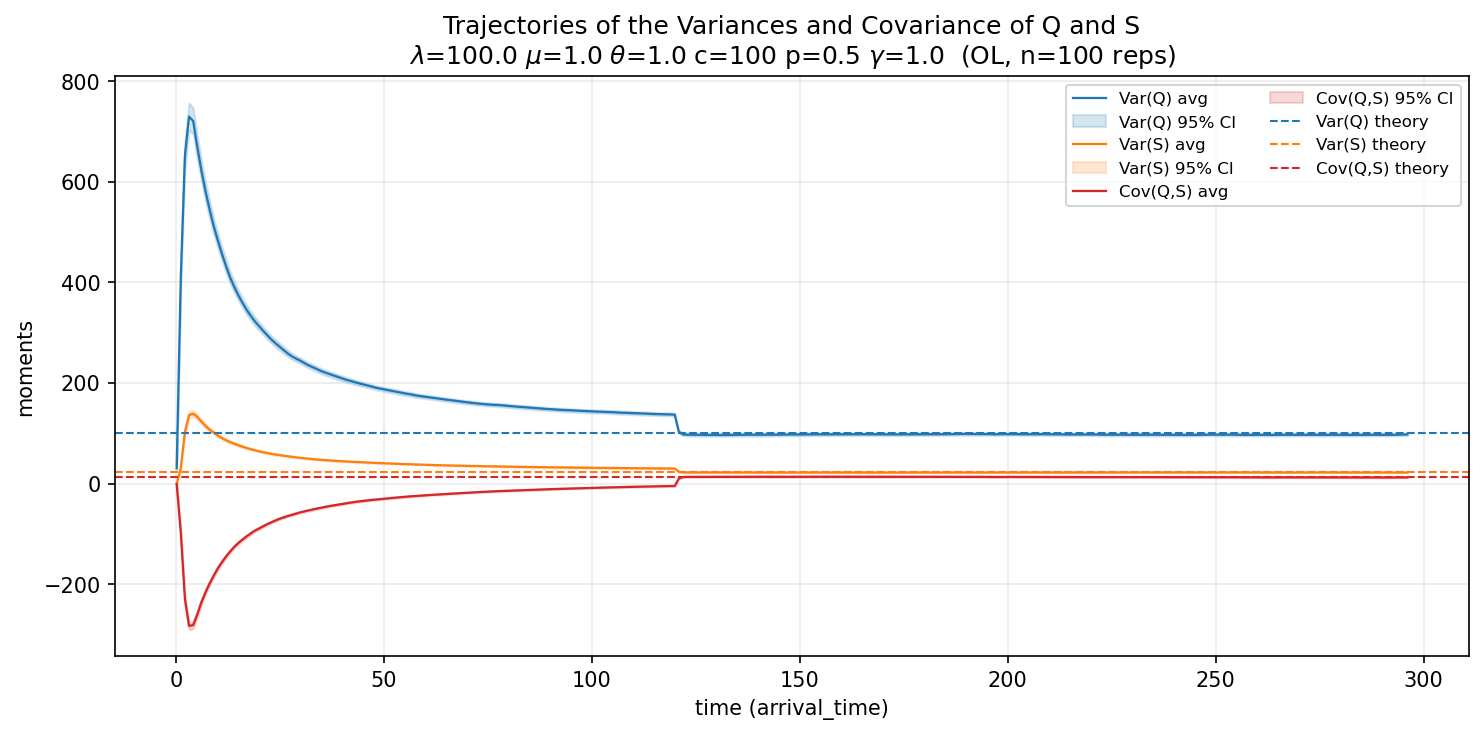}
    \caption{Variance and covariance with $\lambda = 100,\, \mu = 5 , \, \theta = 1,\, p=0.1,\, \gamma = 0.5$ and $c=100$ (UL)}
    \label{fig:var_cov_traj_OL}
\end{figure}

\begin{figure}[t]
    \centering
    \includegraphics[width=0.75\linewidth]{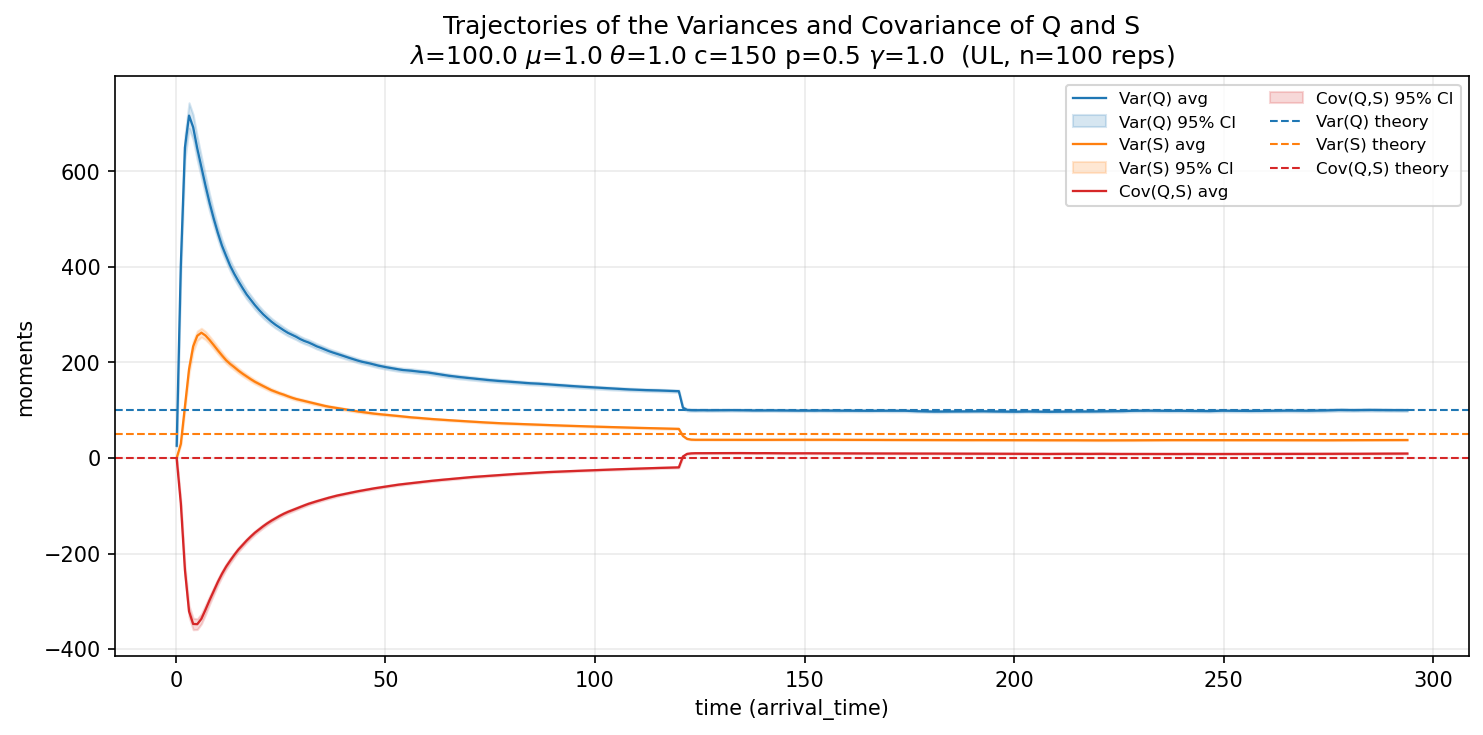}
    \caption{Variance and covariance with $\lambda = 100,\, \mu = 1 , \, \theta = 1,\, p=0.5,\, \gamma = 1,\, c=150$ (boundary UL)}
    \label{fig:var_cov_traj_B_UL}
\end{figure}






\subsection{Trajectory validation: fluid limits}

Figures~\ref{fig:QS-trajectories_UL}, \ref{fig:QS-trajectories_OL} and \ref{fig:QS-trajectories_boundary} show sample paths of $Q(t)$ and $S(t)$ overlayed with their corresponding fluid fixed points. We show 3 plots for the different parameters, Each of them an average of 100 simulation runs. We can also see a small area around the plot lines, which are the confidence intervals. For the upper left plot we simulated an Under-loaded regime with $\lambda = 100,\, \mu = 5 , \, \theta = 1,\, p=0.1,\, \gamma = 0.5$ and $c=100$. Following Theorem~\ref{thm:steady_state_limits}, the steady state limits are $q^* = 20$ and $s^*=80$. In each regime, the trajectories fluctuate around the fluid predictions, with deviations consistent with the diffusion scale. For the upper right plot we simulated an Over-loaded regime, with $\lambda = 100,\, \mu = 1 , \, \theta = 1,\, p=0.5,\, \gamma = 1$ and $c=100$, $Q(t)$ stabilizes around $100$ and $S(t)$ concentrates near $\kappa c = (1/1.5)\times 100 \approx 67$, confirming the accuracy of Theorem~\ref{thm:steady_state_limits}. And for the bottom plot we simulated another Under-loaded regime, but on the boundary (i.e. $\lambda/\mu +\lambda p/\gamma = c$) and we say is near-critical in the UL regime. The simulated parameters were $ \lambda = 100,\, \mu = 1 , \, \theta = 1,\, p=0.5,\, \gamma = 1$ and $c=150$. And notice that, the same way as in the other UL run, $Q(t)$ remains close to $\lambda/\mu$ (for this run 100) while $S(t)$ fluctuates around $c - \lambda p/\gamma$ (for this run: 100).

On the other hand, Figures~\ref{fig:dist_q_s_UL}, \ref{fig:dist_q_s_OL} and
\ref{fig:dist_q_s_boundary} compare the empirical distributions of $Q$ and $S$ from a
single simulation run with the regime-specific normal approximations
$\mathcal N(q^*,v_{qq})$ and $\mathcal N(s^*,v_{ss})$, and with the normal
approximations built from the Gaussian closure of
Section~\ref{sec:gauss-moment-near-boundary}.
 
In the underloaded and overloaded configurations the agreement is not merely close but
exact, and for a structural reason. When $q^*<s^*$ with a gap large relative to
$\sqrt{\sigma^2_D}$ the queue is never blocked, so $Q$ evolves as an $M/M/\infty$ system and is
Poisson$(\lambda/\mu)$; since $M/M/\infty$ is reversible its departure process is
Poisson$(\lambda)$ and independent of the current queue length, and thinning it by $p$
makes the charging population an independent $M/M/\infty$ system with input rate
$\lambda p$ and service rate $\gamma$. Hence $c-S$ is Poisson$(\lambda p/\gamma)$ and
\[
\Var(Q)=\frac{\lambda}{\mu},
\qquad
\Var(S)=\frac{\lambda p}{\gamma},
\qquad
\Cov(Q,S)=0
\]
hold exactly, not merely to diffusion order. Consistently, the Lyapunov equation of
Theorem~\ref{thm:diffusion_UL} returns $v_{qs}=0$ as an identity rather than an
approximation. Symmetrically, when $q^*>s^*$ all servers are busy, the available-server
process is autonomous with birth rate $\gamma(c-s)$ and death rate $p\mu s$, and detailed
balance gives $S\sim\mathrm{Binomial}(c,\kappa)$ with $\kappa=\gamma/(\gamma+p\mu)$;
its variance $c\kappa(1-\kappa)=c\gamma p\mu/(\gamma+p\mu)^2$ is precisely the expression
in Theorem~\ref{thm:diffusion_OL}. Exact computation of the stationary distribution
confirms both: for the parameters of Figure~\ref{fig:dist_q_s_UL} it gives
$\E[Q]=20.0$, $\E[S]=80.0$, $\Var(Q)=20.0$ and $\Var(S)=20.0$ against the
predicted $20$, $80$, $20$ and $20$; for those of Figure~\ref{fig:dist_q_s_OL} it gives
$\E[Q]=100.0$, $\E[S]=66.67$, $\Var(Q)=100.0$ and $\Var(S)=22.223$ against the
predicted $100$, $66.67$, $100$ and $22.222$. The two normal curves in each panel are
therefore indistinguishable, and both track the histogram.
 
The near-critical configuration of Figure~\ref{fig:dist_q_s_boundary}, at which
$c=c^*=150$ and the fluid equilibrium lies exactly on the switching surface, behaves
differently, and the difference is confined to the server process. The queue histogram is
still matched by $\mathcal N(q^*,v_{qq})=\mathcal N(100,100)$. The server histogram is
not: it is centred near $101.8$ rather than at the fluid value $s^*=100$, and is visibly
narrower than $\mathcal N(s^*,\lambda p/\gamma)=\mathcal N(100,50)$. Exact computation
gives $\E[S]=101.764$ and $\Var(S)=38.341$, so the one-sided approximation misplaces the
centre by $1.8$ servers and overstates the variance by $30\%$. Both discrepancies are
recovered by the Gaussian closure, which returns $101.767$ and $37.513$ and whose curve
lies on the histogram in the figure. This is the boundary phenomenon of
Sections~\ref{sec:boundary-diffusion} and~\ref{sec:gauss-moment-near-boundary} seen at
the level of the marginal distributions rather than of the moments: the mean shift is of
diffusion order and therefore invisible to the fluid limit, while the variance is pulled
below its underloaded value by the fraction of time the system spends congested.

\subsection{Diffusion validation: variances and covariance}\label{sec:numerical-validation-covariance}

To validate the diffusion closure, we compute rolling estimates of $\Var(Q(t))$, $\Var(S(t))$, and $\Cov(Q(t),S(t))$ from the post-processed time series and compare them with the theoretical values  $v_{qq},v_{ss},v_{qs}$ obtained from the linear--Gaussian approximation.

Figures~\ref{fig:var_cov_traj_UL}, \ref{fig:var_cov_traj_OL} and \ref{fig:var_cov_traj_B_UL} display the average of 100 simulation runs of empirical rolling moments together with the theoretical constants. The parameters used for the 3 runs are the same (and in the same order) as the ones used for computing the $Q(t)$ and $S(t)$ trajectories. This is, $\lambda = 100,\, \mu = 5 , \, \theta = 1,\, p=0.1,\, \gamma = 0.5$ and $c=100$ for the upper-left;  $\lambda = 100,\, \mu = 1 , \, \theta = 1,\, p=0.5,\, \gamma = 1$ and $c=100$ for the upper-right; and $ \lambda = 100,\, \mu = 1 , \, \theta = 1,\, p=0.5,\, \gamma = 1$ and $c=150$ for the bottom. Note that we also give confidence regions around the average, which are negligible.

As expected, $\Var(S)$, $\Var(Q)$ and $\Cov[Q,S]$ match closely with its diffusion limit. There are slight variations in the simulation outputs in comparison to the theory, but we can see the tendency of convergence of these results.

It is also important to note that in strongly under-loaded regimes shown in Figure \ref{fig:var_cov_traj_UL}, with very fast service and many idle servers, the diffusion approximation tends to overestimate the variance of the active–server process $S$. This reflects the fact that our UL closure neglects some stabilizing feedback between the queue and the server-availability process. Nevertheless, the approximation for the queue-length variance and for $\Cov[Q,S]$ remains accurate, and the resulting staffing rules are still very close to simulation. This matches well with the empirical distribution shown in Figure \ref{fig:dist_q_s_UL}.

The near-critical configuration of Figure~\ref{fig:var_cov_traj_B_UL} corresponds exactly to the boundary condition $q^* = s^* = \lambda/\mu = 100$, with $c = \lambda/\mu + \lambda p/\gamma = 150$ satisfying the exact-staffing condition of the last remark in Section~\ref{subsec:model-construction}. This is therefore the operating point analyzed in Section~\ref{sec:boundary-diffusion}. The simulated $\text{Cov}(Q,S)$ in Figure~\ref{fig:var_cov_traj_B_UL} remains close to zero, consistent with both the UL formula of Theorem~\ref{thm:diffusion_UL} (which predicts $\text{Cov}(Q,S) \approx 0$).

On the other hand, we can see in Figures~\ref{fig:variances_l10} and \ref{fig:variances_l100} how the covariance and variance evolve with respect to $c$ with low and high traffic. Moreover, in these figures we are able to observe how the two different approximations perform within the boundary band: the Gaussian closure and the Gaussian mixture. We can appreciate how for both variances they track accurately the simulations. However, when it comes to approximating the covariance its not. 

Additionally, these figures show how the fluid and diffusion approximation fail around the boundary, but also how we can go around and get an accurate representation of the process via a Gaussian approximation.

The numerical results shown in Figures~~\ref{fig:variances_l10} and \ref{fig:variances_l100} illustrate two important features. First, the normal moment closure transitions smoothly between the overloaded and underloaded limits and approaches the corresponding regime-specific formulas as the system moves away from the boundary. This supports the use of the standardized distance \(z(c)\) to identify the region in which boundary crossings are important. Second, the direct Gaussian mixture captures the overall transition in \(v_{qq}\), but it systematically overestimates \(v_{ss}\) and \(v_{qs}\) near the boundary. The discrepancy is particularly visible for the covariance, which approaches its underloaded value more rapidly under the full moment closure.

This behavior indicates that the boundary moments are not, in general, a simple convex combination of the two stationary Ornstein--Uhlenbeck covariances. The nonlinear expectations involving \(Q\wedge S\) and \((Q-S)^+\) affect the three second moments differently. We therefore use the full bivariate-normal moment closure inside \(\mathcal{B}_2\), while retaining the simpler regime-specific diffusion formulas outside this region.

The numerical comparisons presented here demonstrate consistency between the proposed closure and the one-sided diffusion approximations away from the boundary. They should not, by themselves, be interpreted as validation against the original stochastic system. Such validation requires comparison with steady-state simulation of the Erlang--\(S^*\) process.

\begin{figure}[t]
    \centering
    \includegraphics[width=0.5\linewidth]{Simulation_plots/var_q_l10_t05.png}~\includegraphics[width=0.5\linewidth]{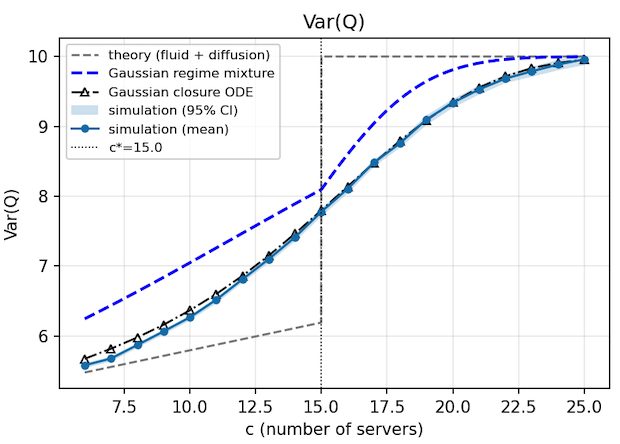}
    \includegraphics[width=0.5\linewidth]{Simulation_plots/var_s_l10_t05.png}~\includegraphics[width=0.5\linewidth]{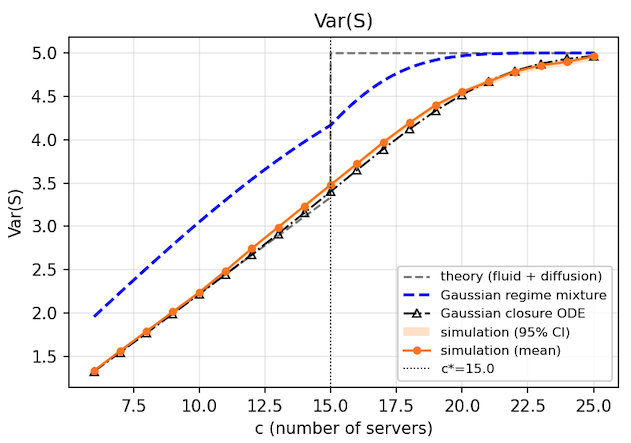}
    \includegraphics[width=0.5\linewidth]{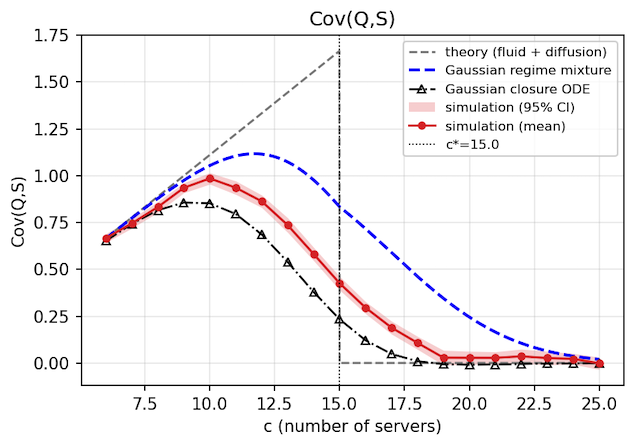}~\includegraphics[width=0.5\linewidth]{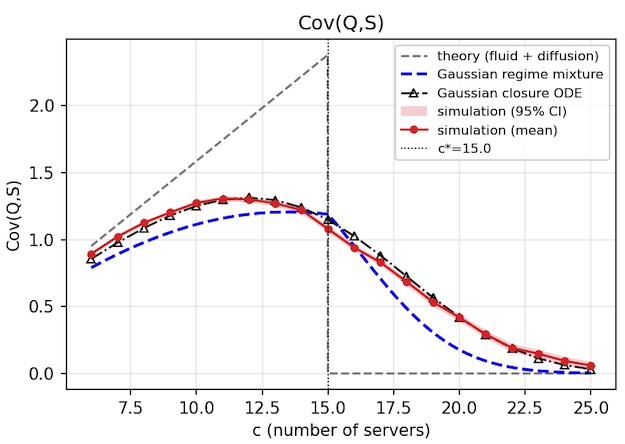}
    \caption{Comparison of Variances and Covariance of $Q$ and $S$ across boundary for $\lambda = 10, \, \mu = 1,\,\, p=0.5,\, \gamma = 1$ changing $\theta = 0.5$ (left) to $\theta = 2$ (right).}
    \label{fig:variances_l10}
\end{figure}

\begin{figure}[t]
    \centering
    \includegraphics[width=0.5\linewidth]{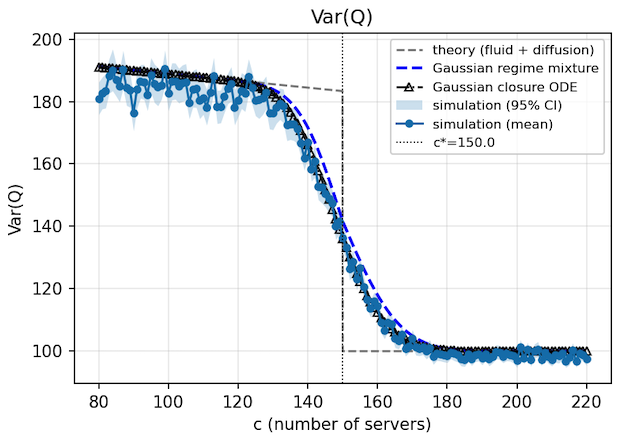}~\includegraphics[width=0.5\linewidth]{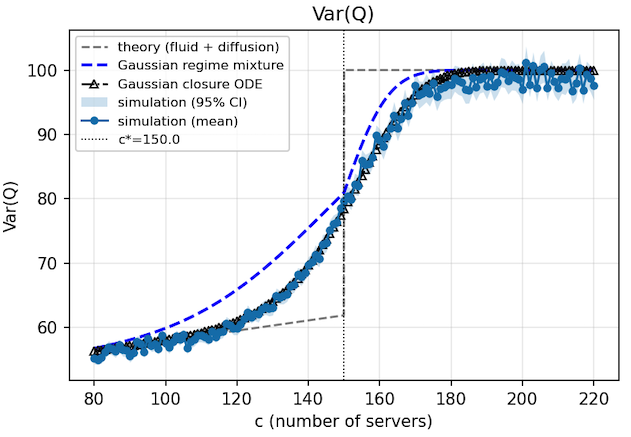}
    \includegraphics[width=0.5\linewidth]{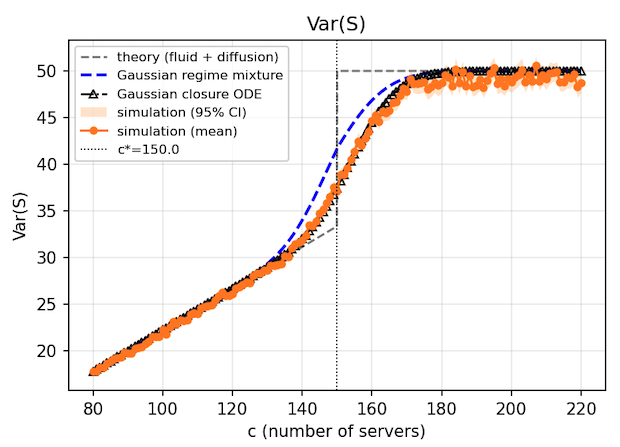}~\includegraphics[width=0.5\linewidth]{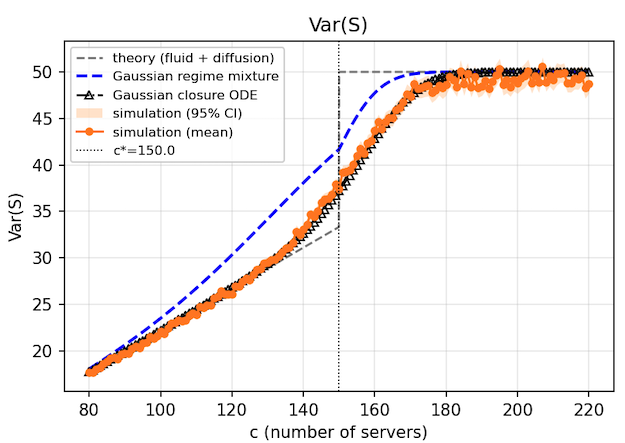}
    \includegraphics[width=0.5\linewidth]{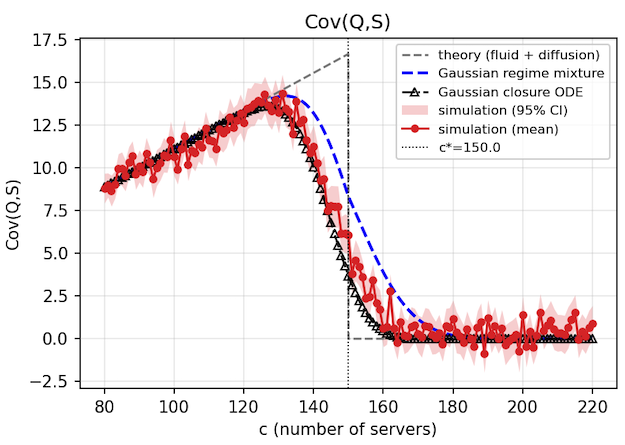}~\includegraphics[width=0.5\linewidth]{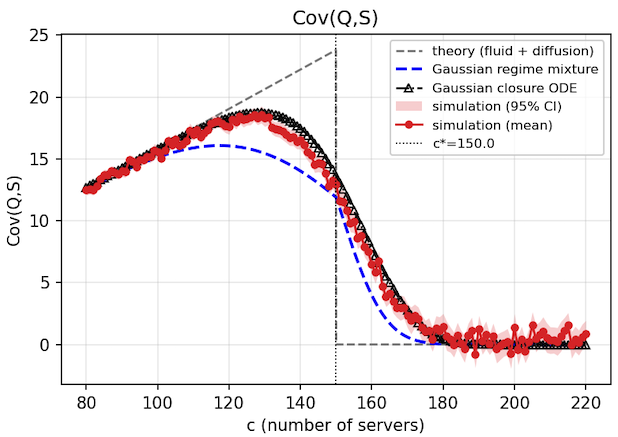}
    \caption{Comparison of Variances and Covariance of $Q$ and $S$ across boundary for $\lambda = 100, \, \mu = 1,\,\, p=0.5,\, \gamma = 1$ changing $\theta = 0.5$ (left) to $\theta = 2$ (right).}
    \label{fig:variances_l100}
\end{figure}

\subsection{Operational validation: staffing for delay and abandonment}

We next compare the staffing levels predicted by the fluid and diffusion rules with the minimum staffing empirically required to meet specified delay and abandonment targets.

\begin{table}[ht]
\centering
\begin{tabular}{ccccccccccc}
\toprule
$\lambda$ & $\mu$ & $\theta$ & $p$ & $\gamma$ &
$\varepsilon_{\text{delay}}$ &
$c_{\text{sim}}$ &
$c_{\text{fluid}}$ &
$c_{\text{fluid}}/c_{\text{sim}}$\% &
$c_{\text{diff}}$ &
$c_{\text{diff}}/c_{\text{sim}}$\% \\
\midrule
80 & 1 & 1 & 0.10 & 0.50 & 0.01 & 121 & 116.81 & 96.54 & 118.79 & 98.18 \\
80 & 1 & 1 & 0.50 & 0.10 & 0.05 & 512 & 494.71 & 96.62 & 516.04 & 100.79 \\
80 & 10 & 1 & 0.50 & 0.50 & 0.10 & 101 & 91.62 & 90.72 & 100.02 & 99.03 \\
\midrule
100 & 1 & 1 & 0.10 & 0.50 & 0.01 & 147 & 143.26 & 97.46 & 145.48 & 98.97 \\
100 & 1 & 1 & 0.50 & 0.10 & 0.05 & 636 & 616.45 & 96.93 & 640.29 & 100.67 \\
100 & 10 & 1 & 0.50 & 0.50 & 0.10 & 124 & 114.05 & 91.98 & 123.44 & 99.55 \\
\midrule
120 & 1 & 1 & 0.10 & 0.50 & 0.01 & 174 & 169.48 & 97.40 & 171.92 & 98.80 \\
120 & 1 & 1 & 0.50 & 0.10 & 0.05 & 759 & 738.02 & 97.24 & 764.14 & 100.68 \\
120 & 10 & 1 & 0.50 & 0.50 & 0.10 & 147 & 136.44 & 92.82 & 146.72 & 99.81 \\
\bottomrule
\end{tabular}
\caption{Comparison of simulated minimal staffing with fluid and diffusion predictions for delay targets.}
\label{tab:delay-staffing}
\end{table}

\begin{table}[ht]
\centering
\begin{tabular}{ccccccccccc}
\toprule
$\lambda$ & $\mu$ & $\theta$ & $p$ & $\gamma$ &
$\varepsilon_{\text{aband}}$ &
$c_{\text{sim}}$ &
$c_{\text{fluid}}$ &
$c_{\text{fluid}}/c_{\text{sim}}$\% &
$c_{\text{diff}}$ &
$c_{\text{diff}}/c_{\text{sim}}$\% \\
\midrule
80 & 1.00 & 1 & 0.50 & 10.00 & 0.01 & 93 & 83.16 & 89.42 & 92.80 & 99.78 \\
80 & 1.00 & 1 & 0.50 & 10.00 & 0.05 & 83 & 79.80 & 96.14 & 82.98 & 99.97 \\
80 & 1.00 & 1 & 0.50 & 10.00 & 0.10 & 77 & 75.60 & 98.18 & 76.71 & 99.62 \\
\midrule
100 & 0.50 & 1 & 0.50 & 0.50 & 0.01 & 314 & 297.00 & 94.59 & 315.30 & 100.41 \\
100 & 0.50 & 1 & 0.50 & 0.50 & 0.05 & 290 & 285.00 & 98.28 & 289.57 & 99.85 \\
100 & 0.50 & 1 & 0.50 & 0.50 & 0.10 & 272 & 270.00 & 99.26 & 271.19 & 99.70 \\
\midrule
120 & 1.00 & 1 & 0.50 & 0.10 & 0.01 & 739 & 712.80 & 96.45 & 740.28 & 100.17 \\
120 & 1.00 & 1 & 0.50 & 0.10 & 0.05 & 691 & 684.00 & 98.99 & 690.08 & 99.87 \\
120 & 1.00 & 1 & 0.50 & 0.10 & 0.10 & 648 & 648.00 & 100.00 & 649.54 & 100.24 \\
\bottomrule
\end{tabular}
\caption{Comparison of simulated minimal staffing with fluid and diffusion predictions for abandonment targets.}
\label{tab:abandon-staffing}
\end{table}

For each target level $\varepsilon$ and parameter configuration, we compute: (i) the minimal staffing $c^{\mathrm{sim}}_\varepsilon$ achieving the target in simulation; (ii) the staffing $c^{\mathrm{fluid}}_\varepsilon$ obtained from the fluid rules in Sections~\ref{sec:staffing}--\ref{sec:abandon}; and (iii) the diffusion-based staffing $c^{\mathrm{diff}}_\varepsilon$ defined implicitly by Theorem~\ref{thm:joint-normal_server_staff-delay} and Theorem~\ref{thm:alpha-joint-normal}. Tables~\ref{tab:delay-staffing} and~\ref{tab:abandon-staffing} summarize these comparisons.

Across all delay targets tested, the diffusion rule consistently narrows the gap between the fluid prediction and the empirical minimum, often matching $c^{\mathrm{sim}}_\varepsilon$ exactly.  For abandonment targets, the implicit bivariate-normal rule produces staffing levels within zero or one servers of the simulation-optimal value, while the fluid rule provides a necessary but not sufficient lower bound, as expected.

These results demonstrate that the diffusion approximation offers substantial accuracy gains over the fluid model in staffing problems, while remaining computationally trivial once the single scalar equation $\alpha(c)=\varepsilon$ is solved.

\section{Conclusion and extensions}
\label{sec:conclusions}

This paper studies the Erlang--$S^*$ queue, in which servers become
temporarily unavailable after service, with drone charging as a motivating
application. We derive the fluid equilibria, interior diffusion limits, and
associated covariance formulas for the joint queue-capacity process $(Q,S)$.
We then use these approximations to construct staffing rules for delay and
abandonment targets and introduce a Gaussian moment closure for the transition
near the critical capacity
\[
c^*=\frac{\lambda}{\mu}+\frac{\lambda p}{\gamma}.
\]

The analysis identifies two distinct effects of stochastic server
availability. First, charging creates a first-order capacity requirement:
in underload, the term $\lambda p/\gamma$ represents the average number of
servers unavailable for charging. Second, fluctuations in $S$ affect the
variance of the gap $Q-S$ and therefore cannot generally be ignored in
delay and abandonment calculations. In the underloaded diffusion,
$\Cov(Q,S)=0$, whereas in overload the covariance is nonzero and its sign
depends on the parameters; specifically,
\[
\Cov(Q,S)<0
\quad\Longleftrightarrow\quad
\mu(1-p)>\gamma+\theta.
\]
At $c^*$, the two interior diffusion approximations do not agree, reflecting
the failure of a one-sided linearization when the system frequently crosses
between $Q<S$ and $Q>S$. The Gaussian closure provides a continuous
approximation through this finite-system transition.

In the configurations reported in the numerical study, the diffusion
staffing rules are substantially closer to the simulation-based benchmarks
than the fluid rules, particularly for abandonment targets. These experiments
support the practical value of incorporating server variability and
covariance, although they do not establish uniform accuracy over all parameter
settings.

The model also clarifies how operational changes affect capacity. Reducing
the charging probability $p$ or increasing the return rate $\gamma$ lowers
the critical capacity $c^*$ and increases effective server availability.
Technologies such as battery swapping or faster charging can therefore be
interpreted as changes in these parameters, although evaluating such policies
directly would require a model that includes their costs and operational
constraints.

Several extensions remain open. Heterogeneous servers, nonexponential charging
times, state-dependent charging decisions, and spatial routing would provide
more realistic models but would generally require higher-dimensional state
descriptions or additional numerical methods. Risk-sensitive staffing is
another natural direction: rather than targeting only mean abandonment or
delay probabilities, one could control tail-based congestion measures as in
\citet{pender2016risk}. Finally, recent work on customer overlap
\citep{boxma2024number,palomo2023maximum,ko2023number,newmaninfinite25,
lucas2025overlapping,boxma2024overlap} suggests studying how stochastic server
availability affects overlap counts and durations.




\bibliographystyle{plainnat}
\bibliography{references}

\newpage


\appendix
\section{Proof Sketches and Technical Details}\label{app:proofs}
\subsection{Proof of theorem \eqref{thm:fluid-limit}}\label{app:fluid-proof}
\begin{proof}
Recall that we have by the strong approximation for Poisson processes (FLLN)
\begin{align*}
\frac{N_A(\lambda n t)}{n} &= \lambda t + \varepsilon_A^n(t), \\
\frac{N_B(n x)}{n} &= x + \varepsilon_B^n(x), \\
\frac{N_{Ab}(n x)}{n} &= x + \varepsilon_{Ab}^n(x), \\
\frac{N_C(n x)}{n} &= x + \varepsilon_C^n(x), \\
\frac{N_B(n x)}{n}\frac{1}{N_B(n x)} \sum_{i=1}^{N_B(n x)} \xi_i &= \bar{\xi} x + \delta^n(x),
\end{align*}
where $\varepsilon_A^n(t),\;\varepsilon_B^n(t),\; \varepsilon_{Ab}^n(t), \varepsilon_C^n(t)$ are martingales of order $O_p(n^{-1/2})$ uniformly on compact intervals. Moreover, the convergence $\delta^n(x) \to 0$ follows from the law of large numbers for random sums: since $\{\xi_i\}$ are i.i.d. Bernoulli$(p)$ independent of $N_B$, we have $\frac{1}{N_B(nx)}\sum_{i=1}^{N_B(nx)} \xi_i \to p$ a.s. by the strong law, and $N_B(nx)/n \to x$ uniformly on compacts by the FLLN for Poisson processes; together these give $\frac{1}{n}\sum_{i \leq N_B(nx)} \xi_i \to px$ uniformly on $[0,T]$ in probability, contributing the term $p\mu \int_0^t \min(\bar{q}(s), \bar{s}(s))\,ds$ to the fluid equation for $\bar{s}(t)$.

Now, subtract the candidate fluid limit and use the Lipschitz properties of
$\min(\cdot,\cdot)$ and $(\cdot)^+$ to obtain
\begin{align*}
    \|(\bar Q^n,\bar S^n)-(\bar q,\bar s)\|(t)
    \le &\|(\bar Q^n(0),\bar S^n(0))-(q_0,s_0)\|
     + o_p(1) \\
     &+ L\!\int_0^t \|(\bar Q^n,\bar S^n)-(\bar q,\bar s)\|(s)\,ds
\end{align*}
for some constant $L>0$. Recall that
\begin{align*}
\bar q(t)&=q_0 + \lambda t
 - \mu\int_0^t \min\{\bar q(s),\bar s(s)\}\,ds
 - \theta\int_0^t \big(\bar q(s)-\bar s(s)\big)^+ ds,\\
\bar s(t)&=s_0 + \gamma\int_0^t \big(c-\bar s(s)\big)\,ds
 - \mu \bar\xi \int_0^t \min\{\bar q(s),\bar s(s)\}\,ds.
\end{align*}

(Note that with this property, we can already conclude by Kurtz 1978 paper that this process converges to the fluid limit)

In other words we can say that:
\begin{align}
0 &\leq -|\bar{Q}^n(t) - \bar q(t) |+| \bar{Q}^n(0) - \bar q(0)|  + \left| \frac{N_A(\lambda n t)}{n} -\lambda t\right| \label{eq:1qn}\\
&\quad + \left| \frac{N_B\left( \mu n \int_0^t \min(\bar{Q}^n(s), \bar{S}^n(s)) ds \right)}{n} -  \mu  \int_0^t \min(\bar{Q}^n(s), \bar{S}^n(s))ds\right| \label{eq:2qn}\\
&\quad + \left| \mu  \int_0^t \min(\bar{Q}^n(s), \bar{S}^n(s))ds - \mu \int_0^t \min (\bar q(s) , \bar s(s))ds\right| \label{eq:3qn_linear_drift}\\
&\quad + \left| \frac{N_{Ab}\left( \theta n \int_0^t (\bar{Q}^n(s) - \bar{S}^n(s))^+ ds \right)}{n} - \theta \int_0^t (\bar Q^n(s) - \bar S^n(s))^+ds \right| \label{eq:4qn} \\
&\quad + \left| \theta \int_0^t (\bar Q^n(s) - \bar S^n(s))^+ds - \theta \int_0^t (\bar q(s) - \bar s(s))^+ds \right|,\label{eq:5qn_linear_drift}
\end{align}
with \eqref{eq:1qn}, \eqref{eq:2qn}, \eqref{eq:4qn} corresponding to the assumptions and the convergence of scaled Poisson processes by strong approximation. On the other hand, we know that \eqref{eq:3qn_linear_drift}, \eqref{eq:5qn_linear_drift} contain the functions $\min(\cdot)$ and $(\cdot)^+$ that are 1-Lipschitz. Therefore 
\begin{align*}
    \left| \mu  \int_0^t \min(\bar{Q}^n(s), \bar{S}^n(s))ds - \mu \int_0^t \min (\bar q(s) , \bar s(s))ds\right| &\leq \mu \int_0^t | \bar Q^n(s)-\bar q(s)|\\ &\quad \;+ |\bar S^n(s)-\bar s(s)|ds \\
    \left| \theta \int_0^t (\bar Q^n(s) - \bar S^n(s))^+ds - \theta \int_0^t (\bar q(s) - \bar s(s))^+ds \right| &\leq \theta \int_0^t | \bar Q^n(s)-\bar q(s)| \\ &\quad \;+ |\bar S^n(s)-\bar s(s)|ds.
\end{align*}

Analogously, we have that

\begin{align}
0 &\leq -|\bar{S}^n(t) - \bar s(t)| + |\bar{S}^n(0) - \bar s(0)| \label{eq:1sn}\\
&\quad +\left| \frac{N_C\left( \gamma n \int_0^t (c - \bar{S}^n(s)) ds \right)}{n}
- \gamma\int_0^t (c-\bar S^n(s))ds\right| \label{eq:2sn}\\
&\quad +\left| \gamma\int_0^t (c-\bar S^n(s))ds - \gamma \int_0^t (c-\bar s(s))ds \right| \label{eq:3sn_linear_drift}\\
&\quad + \left|\frac{1}{n} \sum_{i=1}^{N_B(\mu n \int_0^t \min(\bar{Q}^n(s), \bar{S}^n(s)) ds)} \xi_i - \mu\bar\xi\int_0^t \min (\bar Q^n(s),\bar S^n(s)) ds \right| \label{eq:4sn}\\
&\quad + \left|\mu\bar\xi\int_0^t \min (\bar Q^n(s),\bar S^n(s))ds - \mu \bar\xi \int_0^t \min (\bar q(s),\bar s(s)) ds \right|,\label{eq:5sn_linear_drift}
\end{align}
with \eqref{eq:1sn}, \eqref{eq:2sn}, \eqref{eq:4sn} corresponding to the assumptions, the convergence of scaled Poisson processes by strong approximation and the convergence of the random sum term which are shown how are handled above. Moreover, the same as before, we know that \eqref{eq:3sn_linear_drift}, \eqref{eq:5sn_linear_drift} contain the functions $\min(\cdot)$ and a linear difference that are 1-Lipschitz. Hence

\begin{align*}
    \left| \gamma\int_0^t (c-\bar S^n(s))ds - \gamma \int_0^t (c-\bar s(s))ds \right| &\leq \gamma \int_0^t  |\bar S^n(s)-\bar s(s)|ds \\
    \left|\mu\bar\xi\int_0^t \min (\bar Q^n(s),\bar S^n(s))ds - \mu \bar\xi \int_0^t \min (\bar q(s),\bar s(s)) ds \right| &\leq \mu\bar\xi \int_0^t | \bar Q^n(s)-\bar q(s)| \\&\quad +\; |\bar S^n(s)-\bar s(s)|ds.
\end{align*}

Now, notice that both inequalities depend on this term, so if we sum them and then bound the sum, this should give us the result we want. Let 
\[
E_n(t):= |\bar Q^n(t)-\bar q(t)|+|\bar S^n(t)-\bar s(t)|,
\]
and collect all Poisson remainders into
\begin{align*}
    A_n(T) := \sup_{0\leq u\leq T} \Bigg( &\left| \frac{N_A(\lambda n u)}{n} -\lambda u\right| + \left| \frac{N_B(\cdot)}{n} - \cdot\right|+ \left| \frac{N_{Ab}(\cdot)}{n} - \cdot\right| \\ &+ \left| \frac{N_C(\cdot)}{n} - \cdot\right| + \left| \frac{1}{n} \sum_{i\leq N_B(\cdot)} \xi_i-\bar\xi \right| \Bigg).
\end{align*}

Then, we know that there exists an $L>0$ (independent of $n$) such that for all $t\in [0,T]$,
\[
E_n(t) \leq E_n(0) + A_n(T) + L\int_o^t E_n(s)ds.
\]
Now, lets define $F_n(t) : = \sup_{0\leq u\leq t} E_n(u)$, then we can deduce that
\[
F_n(t) \leq E_n(0) + A_n(T) + L\int_o^t F_n(s)ds.
\]
With this we can apply Gr\"onwall's inequality
\[
F_n(t) \leq (E_n(0) + A_n(T)) e^{Lt} \; \Rightarrow \;  \sup_{0\leq u\leq t} E_n(u) \leq (E_n(0) + A_n(T)) e^{Lt}
\]
The initial conditions converge by assumption and $A_n(T)$ converges as we showed before, so the right-hand side vanishes in probability, proving the claim.
\end{proof}

A standard Kurtz FLLN argument for time-changed Poisson processes and random sums applied to \eqref{eq:Q-int}--\eqref{eq:S-int} yields \eqref{eq:fluid-q}--\eqref{eq:fluid-s}. Uniqueness follows from Lipschitz continuity of the drift.

\subsection{Proof of Theorem~\ref{thm:diffusion_limit}}
\label{app:diffusion-proof}

\begin{proof}
Let
\[
X^n(t):=
\begin{pmatrix}
Q^n(t)\\
S^n(t)
\end{pmatrix},
\qquad
\bar X^n(t):=\frac{X^n(t)}{n}.
\]
Under the many-server scaling, the transition intensities have the
density-dependent form
\[
\lambda_k^n(nx)=n\beta_k(x),
\qquad k=1,\ldots,K,
\]
where $\ell_k\in\mathbb R^2$ is the jump vector associated with event type
$k$. Using independent unit-rate Poisson processes $N_k$, the process admits
the random time-change representation
\[
X^n(t)
=
X^n(0)
+
\sum_{k=1}^K
\ell_k
N_k\left(
n\int_0^t\beta_k\bigl(\bar X^n(u)\bigr)\,du
\right).
\]
Consequently,
\begin{equation}
\bar X^n(t)
=
\bar X^n(0)
+
\int_0^t b\bigl(\bar X^n(u)\bigr)\,du
+
M^n(t),
\label{eq:scaled-martingale-decomposition}
\end{equation}
where
\[
b(x):=\sum_{k=1}^K\ell_k\beta_k(x)
\]
is the fluid drift and $M^n$ is the associated martingale.

The fluid law of large numbers gives
\[
\sup_{0\leq t\leq T}
\left\|\bar X^n(t)-x(t)\right\|
\xrightarrow{p}0.
\]
By assumption,
\[
\delta_T:=
\inf_{0\leq t\leq T}|q(t)-s(t)|>0.
\]
Hence, with probability tending to one, $\bar X^n$ and $x$ remain on the same
face throughout $[0,T]$. On this face the drift is affine, with Jacobian
$J_r$, where $r\in\{\mathrm{UL},\mathrm{OL}\}$ denotes the active regime.

Subtracting the fluid equation
\[
x(t)=x(0)+\int_0^t b(x(u))\,du
\]
from \eqref{eq:scaled-martingale-decomposition} and multiplying by $\sqrt n$
gives
\begin{equation}
\widehat X^n(t)
=
\widehat X^n(0)
+
\int_0^t
J_r\widehat X^n(u)\,du
+
\widehat M^n(t)
+
o_p(1),
\label{eq:scaled-diffusion-decomposition}
\end{equation}
uniformly on $[0,T]$, where
\[
\widehat M^n(t):=\sqrt n\,M^n(t).
\]

The predictable quadratic variation of $\widehat M^n$ is
\[
\left\langle\widehat M^n\right\rangle(t)
=
\int_0^t
\sum_{k=1}^K
\ell_k\ell_k^\top
\beta_k\bigl(\bar X^n(u)\bigr)\,du.
\]
By fluid convergence and continuity of the rate functions,
\[
\left\langle\widehat M^n\right\rangle(t)
\xrightarrow{p}
\int_0^t\Sigma_r(u)\,du,
\qquad
\Sigma_r(u)
:=
\sum_{k=1}^K
\ell_k\ell_k^\top\beta_k(x(u)).
\]
Moreover, the jumps of $\widehat M^n$ are of order $n^{-1/2}$, so the
Lindeberg condition holds. The martingale functional central limit theorem
therefore yields
\[
\widehat M^n
\Rightarrow
\int_0^\cdot\Sigma_r^{1/2}(u)\,dW(u)
\]
in $D([0,T],\mathbb R^2)$.

The solution map associated with the linear integral equation in
\eqref{eq:scaled-diffusion-decomposition} is continuous. Hence,
\[
\widehat X^n\Rightarrow\widehat X,
\]
where $\widehat X$ is the unique strong solution of
\[
d\widehat X(t)
=
J_r\widehat X(t)\,dt
+
\Sigma_r^{1/2}(t)\,dW(t).
\]

If $x(t)\equiv x^*$ is an interior fluid equilibrium, then
$\Sigma_r(t)\equiv\Sigma_r$ and the limit is the Ornstein--Uhlenbeck process
\[
d\widehat X(t)
=
J_r\widehat X(t)\,dt
+
\Sigma_r^{1/2}\,dW(t).
\]
Because $J_r$ is Hurwitz in both interior regimes, this process has a unique
stationary Gaussian distribution with covariance matrix $V_r$ satisfying
\[
J_rV_r+V_rJ_r^\top+\Sigma_r=0.
\]
\end{proof}

\subsection{Covariance and Variance Computaions for Underloaded Regime, proof of Theorem \ref{thm:diffusion_UL}.}\label{app:underloaded-proof}
\begin{proof}
	Here the active faces are $\min(q,s)=q$ and $(q-s)^+=0$. The fluid fixed point is
$q^*=\lambda/\mu$ and $s^*=c-\lambda p/\gamma$.
Linearizing the fluid drift on this face yields
\[
J_{\mathrm{UL}}\;=\;
\begin{pmatrix}
-\mu & 0\\[2pt]
-\,p\mu & -\,\gamma
\end{pmatrix}.
\]
Each primitive event contributes its increment vector and mean rate at the fixed point:
\[
\begin{array}{c|c|c}
\text{event} & \text{increment } \nu_e & \text{rate at }(q^*,s^*) \\\hline
\text{arrival} & (1,0) & \lambda \\
\text{unmarked completion} & (-1,0) & (1-p)\,\mu\,q^* \\
\text{marked completion} & (-1,-1) & p\,\mu\,q^* \\
\text{return-from-charge} & (0,1) & \gamma\,(c-s^*)
\end{array}
\]
Hence the diffusion matrix is
\[
\Sigma_{\mathrm{UL}}
=\lambda
\begin{pmatrix}1&0\\0&0\end{pmatrix}
+(1-p)\mu q^*\begin{pmatrix}1&0\\0&0\end{pmatrix}
+p\mu q^*\begin{pmatrix}1&1\\[2pt]1&1\end{pmatrix}
+\gamma(c-s^*)\begin{pmatrix}0&0\\[2pt]0&1\end{pmatrix}.
\]
Using the mean balance $\gamma(c-s^*)=p\mu q^*$ and $q^*=\lambda/\mu$ simplifies this to
\[
\Sigma_{\mathrm{UL}}
=
\begin{pmatrix}
2\lambda & p\mu q^*\\[2pt]
p\mu q^* & 2p\mu q^*
\end{pmatrix}.
\]
The stationary covariance $V_{\mathrm{UL}}=\operatorname{Cov}_\pi(\hat X)$ solves the Lyapunov equation
\[
J_{\mathrm{UL}}\,V_{\mathrm{UL}}+V_{\mathrm{UL}}\,J_{\mathrm{UL}}^\top+\Sigma_{\mathrm{UL}}=0,
\]
whose unique solution is
\[
V_{\mathrm{UL}}
=
\begin{pmatrix}
\lambda/\mu & 0\\[2pt]
0 & \lambda p/\gamma
\end{pmatrix}.
\]
Therefore
\[
\Var[Q]\approx v_{qq}^*=\frac{\lambda}{\mu},\qquad
\Var[S]\approx v_{ss}^*=\frac{\lambda p}{\gamma},\qquad
\operatorname{Cov}[Q,S]\approx 0,
\]
\end{proof}

\subsection{Covariance and Variance Computations for Overloaded Regime, proof of Theorem \ref{thm:diffusion_OL}}\label{app:overloaded-proof}
\begin{proof}
	Here the active faces are $\min(q,s)=s$ and $(q-s)^+=q-s$. The fluid fixed point satisfies
$s^*=\kappa c$, $\kappa=\gamma/(\gamma+p\mu)$, and $q^*=\frac{\lambda}{\theta}
+s^*\!\left(1-\frac{\mu}{\theta}\right)$.
Linearizing on this face gives
\[
J_{\mathrm{OL}}\;=\;
\begin{pmatrix}
-\,\theta & -(\mu-\theta)\\[2pt]
0 & -(\gamma+p\mu)
\end{pmatrix}.
\]
Event increments and rates at $(q^*,s^*)$ are now
\[
\begin{array}{c|c|c}
\text{event} & \text{increment } \nu_e & \text{rate at }(q^*,s^*) \\\hline
\text{arrival} & (1,0) & \lambda \\
\text{unmarked completion} & (-1,0) & (1-p)\,\mu\,s^* \\
\text{marked completion} & (-1,-1) & p\,\mu\,s^* \\
\text{abandonment} & (-1,0) & \theta\,(q^*-s^*) \\
\text{return-from-charge} & (0,1) & \gamma\,(c-s^*)=p\mu s^*
\end{array}
\]
which yields
\begin{align*}
    \Sigma_{\mathrm{OL}} &=\lambda\begin{pmatrix}1&0\\0&0\end{pmatrix}
+(1-p)\mu s^*\begin{pmatrix}1&0\\0&0\end{pmatrix}
+p\mu s^*\begin{pmatrix}1&1\\[2pt]1&1\end{pmatrix}\\
&\quad +\theta(q^*-s^*)\begin{pmatrix}1&0\\0&0\end{pmatrix}
+\gamma(c-s^*)\begin{pmatrix}0&0\\[2pt]0&1\end{pmatrix}.
\end{align*}
Using the mean identities $q^*-s^*=(\lambda-\mu s^*)/\theta$ and $\gamma(c-s^*)=p\mu s^*$,
this simplifies to
\[
\Sigma_{\mathrm{OL}}
=
\begin{pmatrix}
2\lambda & p\mu s^*\\[2pt]
p\mu s^* & 2p\mu s^*
\end{pmatrix}.
\]
The stationary covariance $V_{\mathrm{OL}}$ solves
\[
J_{\mathrm{OL}}\,V_{\mathrm{OL}}+V_{\mathrm{OL}}\,J_{\mathrm{OL}}^\top+\Sigma_{\mathrm{OL}}=0,
\]
with the unique solution
\[
V_{\mathrm{OL}}
=
\begin{pmatrix}
\displaystyle\frac{\lambda + (\theta-\mu)v_{qs}^*}{\theta}  & \displaystyle \frac{p\mu s^*-(\mu-\theta)\,v_{ss}^*}{\theta+\gamma+p\mu}\\[10pt]
\displaystyle \frac{p\mu s^*-(\mu-\theta)\,v_{ss}^*}{\theta+\gamma+p\mu} &
\displaystyle \frac{p\mu}{\gamma+p\mu}\,s^*
\end{pmatrix},
\]
with
\[
\qquad
v_{ss}^*=\frac{p\mu}{\gamma+p\mu}\,s^*=\frac{c\,\gamma p\mu}{(\gamma+p\mu)^2}.
\]
Equivalently,
\[
\Var[Q]\approx \frac{\lambda}{\theta} + \frac{\theta - \mu }{\theta}\cdot\frac{c\,\gamma p\mu}{(\gamma+p\mu)^2}\cdot
\frac{\gamma+\theta+p\mu-\mu}{\theta+\gamma+p\mu},
\]
\[
\Var[S]\approx \frac{c\,\gamma p\mu}{(\gamma+p\mu)^2},\quad 
\operatorname{Cov}[Q,S]\approx
\frac{c\,\gamma p\mu}{(\gamma+p\mu)^2}\cdot
\frac{\gamma+\theta+p\mu-\mu}{\theta+\gamma+p\mu}.
\]
\end{proof}

\subsection{Staffing Quadratics}
The joint-normal delay condition \eqref{eq:delay-prob} reduces to a quadratic in $c$ in both regimes; we report the larger root for target $\varepsilon$.
\subsubsection{Deterministic Servers, proof of Theorem \ref{thm:determ_server_staff-delay}}\label{app:deterministic-staffing-proof}
\begin{proof}
    In the underloaded case, we have 
    \begin{eqnarray}
        \overline{\Phi} \left( \frac{s^* - q^*}{\sqrt{ q^*}}  \right) &=& \epsilon \\
       \frac{s^* - q^*}{\sqrt{ q^*}}  &=& \overline{\Phi}^{-1} \left( \epsilon \right) \\
       s^* &=& q^* + \sqrt{ q^*} \ \overline{\Phi}^{-1} \left( \epsilon \right) \\
       c - \frac{\lambda p}{\gamma} &=& \frac{\lambda}{\mu} + \sqrt{ \frac{\lambda}{\mu} }  \ \overline{\Phi}^{-1} \left( \epsilon \right)
    \end{eqnarray}
    This implies that 
    \begin{eqnarray}
        c_{\epsilon} \approx \frac{\lambda p}{\gamma} +  \frac{\lambda}{\mu} + \sqrt{ \frac{\lambda}{\mu} }  \ \overline{\Phi}^{-1} \left( \epsilon \right) .
    \end{eqnarray}
    
    Note that this is quite different than the case where the servers are not stochastic.  We see a shift in the number of servers needed of $\frac{\lambda p}{\gamma}$.  
    
    In the overloaded case, we have 
            
    \begin{eqnarray}
       s^* &=& q^* + \sqrt{ q^*} \ \overline{\Phi}^{-1} \left( \epsilon \right) \\
       \frac{\gamma }{\gamma + p \mu } c &=& \frac{\lambda - \frac{\gamma \mu c}{\gamma + p \mu } }{\theta} + \frac{\gamma c}{\gamma + p \mu }  + \sqrt{ \frac{\lambda - \frac{\gamma \mu c}{\gamma + p \mu } }{\theta} + \frac{\gamma c}{\gamma + p \mu }  }  \ \overline{\Phi}^{-1} \left( \epsilon \right) \\
        \frac{\gamma \mu c}{\theta (\gamma + p \mu) }   &=& \frac{\lambda}{\theta} + \sqrt{ \frac{\lambda - \frac{\gamma \mu c}{\gamma + p \mu } }{\theta} + \frac{\gamma c}{\gamma + p \mu }  }  \ \overline{\Phi}^{-1} \left( \epsilon \right)  \\
        c &=& \frac{\lambda (\gamma + p \mu)}{\gamma \mu } + \frac{\theta (\gamma + p \mu)}{\gamma \mu } \sqrt{ \frac{\lambda - \frac{\gamma \mu c}{\gamma + p \mu } }{\theta} + \frac{\gamma c}{\gamma + p \mu }  }  \ \overline{\Phi}^{-1} \left( \epsilon \right) \\
        c &=& \frac{\lambda (\gamma + p \mu)}{\gamma \mu } + \frac{\theta (\gamma + p \mu)}{\gamma \mu } \sqrt{ \frac{\lambda}{\theta} + \frac{ \gamma  c (\theta - \mu) }{\theta(\gamma + p \mu)} } \ \overline{\Phi}^{-1} \left( \epsilon \right)
    \end{eqnarray}
    
    Solving for $c$, simplifying, and writing down $z_\epsilon = \overline{\Phi}^{-1}(\epsilon)$, we get the desired expression
    \[
    c_\epsilon \;=\;
    \frac{\gamma + p \mu}{2 \gamma \mu^{2}}
    \left[
    2 \lambda \mu \;-\; \theta(\mu - \theta)\,z_\epsilon^{2}
    \;+\;
    \sqrt{\Big(2 \lambda \mu - \theta(\mu - \theta)\,z_\epsilon^{2}\Big)^{2}
    - 4 \lambda \mu (\lambda - \theta\, z_\epsilon^{2})}
    \right].
    \]
\end{proof}
\subsubsection{Joint-normal Queue and Server Distribution for delay approximation, proof of Theorem \ref{thm:joint-normal_server_staff-delay}}\label{app:joint-normal-correlated-proof}
\begin{proof}
The joint-normal delay condition is
\begin{align*}
    \Pr(Q\ge S) \;&=\; \Phi\!\left(\frac{q^*-s^*}{\sqrt{\operatorname{Var}[Q]+\operatorname{Var}[S]-2\operatorname{Cov}[Q,S]}}\right)\;=\;\epsilon \\ \Longleftrightarrow\quad s^* \;&=\; q^* \;+\; z_\epsilon\,\sqrt{v_{qq}+v_{ss}-2v_{qs}}.
\tag{$\star$}
\end{align*}

\noindent\emph{(a) Underloaded.}
Insert $q^*=\lambda/\mu$, $s^*=c-\lambda p/\gamma$, and the underloaded closures
$v_{qq}=\lambda/\mu$, $v_{ss}=\lambda p/\gamma$, $v_{qs}=0$ into $(\star)$:
\begin{equation}
\label{eq:c-under-eqn}
c_\epsilon \;=\;
\frac{\lambda p}{\gamma} \;+\; \frac{\lambda}{\mu}
\;+\; z_\epsilon\,
\sqrt{\,\frac{\lambda}{\mu}+\frac{\lambda p}{\gamma}\,}\,
\end{equation}
which is the recommended staffing level. 

\medskip
\noindent\emph{(b) Overloaded.}
Write $s^*=\kappa c$, $q^*=\frac{\lambda}{\theta}+\kappa c(1-\frac{\mu}{\theta})$.
Use the OU second–order terms listed in the theorem; then
\[
v_{qq}+v_{ss}-2v_{qs}
\;=\;
\frac{\lambda}{\theta} \;+\; U\,c,
\quad
U:=\kappa\!\left(1 - \frac{\mu}{\theta}\right)
- \frac{2\,\gamma p\mu}{(\gamma+p\mu)^2}\cdot
\frac{\gamma+\theta+p\mu-\mu}{\theta+\gamma+p\mu}.
\]
Plug into $(\star)$ and isolate the square root:
\[
\frac{\mu\kappa c - \lambda}{\theta}
\;=\;
z_\epsilon\,\sqrt{\;\frac{\lambda}{\theta} + U\,c\;}.
\]
Square and rearrange to obtain the quadratic
\[
\mu^2\kappa^2\,c^2
\;-\; \Bigl(2\mu\kappa\lambda + z_\epsilon^2\,\theta^2 U\Bigr)c
\;+\; \Bigl(\lambda^2 - z_\epsilon^2\,\theta\lambda\Bigr)
\;=\; 0,
\]
whose larger root yields \eqref{eq:c-over-cov-final}:
\[
c_\epsilon
\;=\;
\frac{\,2\mu\kappa\lambda \;+\; z_\epsilon^2\,\theta^2\,U\;+\;
\sqrt{\bigl(2\mu\kappa\lambda + z_\epsilon^2\,\theta^2 U\bigr)^2
- 4\,\mu^2\kappa^2\,\bigl(\lambda^2 - z_\epsilon^2\,\theta\lambda\bigr)}\;}
{2\,\mu^2\kappa^2}
\]
\end{proof}

\subsection{Exact forward equations for the first two moments}
\label{app:exact-moment-equations}

We derive the transient moment equations used in the Gaussian closure.
Let
\[
B(t):=Q(t)\wedge S(t),
\qquad
W(t):=(Q(t)-S(t))^+,
\]
denote the number of busy servers and the number of waiting customers,
respectively. To simplify notation, the time argument is suppressed whenever
there is no ambiguity.

For any sufficiently integrable function \(f\), recall Dynkin's formula gives
\[
\frac{d}{dt}\E[f(Q(t),S(t))]
=
\E[\mathcal{L}f(Q(t),S(t))].
\]

Define the first moments
\[
m_Q(t):=\E[Q(t)],
\qquad
m_S(t):=\E[S(t)].
\]
Taking \(f(q,s)=q\) and \(f(q,s)=s\) in
\eqref{eq:generator} yields
\begin{align}
\dot m_Q
&=
\lambda-\mu\E[B]-\theta\E[W],
\label{eq:exact-mean-Q}
\\
\dot m_S
&=
\gamma(c-m_S)-p\mu\E[B].
\label{eq:exact-mean-S}
\end{align}

Next, define the raw second moments
\[
M_{QQ}(t):=\E[Q(t)^2],
\qquad
M_{SS}(t):=\E[S(t)^2],
\qquad
M_{QS}(t):=\E[Q(t)S(t)].
\]
Applying the generator to \(q^2\), \(s^2\), and \(qs\), respectively, gives
\begin{align}
\dot M_{QQ}
={}&
2\lambda m_Q+\lambda
-2\mu\E[QB]+\mu\E[B]
\nonumber\\
&-2\theta\E[QW]+\theta\E[W],
\label{eq:exact-raw-QQ}
\\
\dot M_{SS}
={}&
2\gamma\bigl(cm_S-M_{SS}\bigr)
+\gamma(c-m_S)
\nonumber\\
&-2p\mu\E[SB]+p\mu\E[B],
\label{eq:exact-raw-SS}
\\
\dot M_{QS}
={}&
\lambda m_S
-\mu\E[SB]
-p\mu\E[QB]
+p\mu\E[B]
\nonumber\\
&-\theta\E[SW]
+\gamma\bigl(cm_Q-M_{QS}\bigr).
\label{eq:exact-raw-QS}
\end{align}

To see the origin of the term \(p\mu\E[B]\) in
\eqref{eq:exact-raw-QS}, observe that a service completion followed by
charging produces the jump
\[
(\Delta Q,\Delta S)=(-1,-1).
\]
Hence,
\[
(Q-1)(S-1)-QS=-Q-S+1,
\]
and the \(+1\) produces the term \(p\mu\E[B]\).

Define
\[
v_{qq}(t):=\Var(Q(t)),
\qquad
v_{ss}(t):=\Var(S(t)),
\qquad
v_{qs}(t):=\Cov(Q(t),S(t)).
\]
Since
\[
v_{qq}=M_{QQ}-m_Q^2,
\qquad
v_{ss}=M_{SS}-m_S^2,
\qquad
v_{qs}=M_{QS}-m_Qm_S,
\]
differentiating and substituting
\eqref{eq:exact-mean-Q}--\eqref{eq:exact-raw-QS} gives
\begin{align}
\dot v_{qq}
={}&
\lambda+\mu\E[B]+\theta\E[W]
-2\mu\Cov(Q,B)
-2\theta\Cov(Q,W),
\label{eq:exact-var-Q}
\\
\dot v_{ss}
={}&
\gamma(c-m_S)+p\mu\E[B]
-2\gamma v_{ss}
-2p\mu\Cov(S,B),
\label{eq:exact-var-S}
\\
\dot v_{qs}
={}&
-\mu\Cov(S,B)
-\theta\Cov(S,W)
-\gamma v_{qs}
\nonumber\\
&-p\mu\Cov(Q,B)
+p\mu\E[B].
\label{eq:exact-cov-QS}
\end{align}
Equations
\eqref{eq:exact-mean-Q}--\eqref{eq:exact-cov-QS} are exact.
They are not closed because they involve nonlinear expectations and
covariances containing \(B=Q\wedge S\) and \(W=(Q-S)^+\).

\subsection{Bivariate-Gaussian moment closure}
\label{app:gaussian-moment-closure}

To close the moment equations, we approximate the transient joint
distribution of \((Q(t),S(t))\) by
\begin{equation}
\begin{pmatrix}
Q(t)\\
S(t)
\end{pmatrix}
\approx
\mathcal{N}
\left(
\begin{pmatrix}
m_Q(t)\\
m_S(t)
\end{pmatrix},
\begin{pmatrix}
v_{qq}(t)&v_{qs}(t)\\
v_{qs}(t)&v_{ss}(t)
\end{pmatrix}
\right).
\label{eq:gaussian-closure-assumption}
\end{equation}
This is a moment-closure approximation rather than a claim that the
finite-system stationary distribution is exactly Gaussian.

Define the gap
\[
D:=Q-S.
\]
Under \eqref{eq:gaussian-closure-assumption},
\[
D\approx\mathcal{N}(m_D,\sigma_D^2),
\]
where
\begin{equation}
m_D:=m_Q-m_S,
\qquad
\sigma_D^2:=v_{qq}+v_{ss}-2v_{qs}.
\label{eq:gap-mean-variance}
\end{equation}
Let
\begin{equation}
a:=\frac{m_D}{\sigma_D},
\qquad
P:=\Phi(a),
\label{eq:standardized-gap}
\end{equation}
where \(\phi\) and \(\Phi\) are the standard normal density and
distribution functions. The standard positive-part identity gives
\begin{equation}
w
:=
\E[W]
=
\E[D^+]
\approx
\sigma_D\phi(a)+m_D\Phi(a).
\label{eq:gaussian-expected-waiting}
\end{equation}
Because
\[
B=Q\wedge S=Q-(Q-S)^+=Q-W,
\]
we obtain
\begin{equation}
b
:=
\E[B]
\approx
m_Q-w.
\label{eq:gaussian-expected-busy}
\end{equation}
Equivalently,
\[
b
=
m_Q(1-P)+m_SP-\sigma_D\phi(a).
\]

We next compute the covariance terms required by
\eqref{eq:exact-var-Q}--\eqref{eq:exact-cov-QS}. If \(X\) and \(D\) are
jointly Gaussian, then Stein's identity implies
\begin{equation}
\Cov(X,D^+)
=
\Cov(X,D)\Phi\left(\frac{\E[D]}{\sqrt{\Var(D)}}\right).
\label{eq:stein-positive-part}
\end{equation}
Since
\[
\Cov(Q,D)=v_{qq}-v_{qs},
\qquad
\Cov(S,D)=v_{qs}-v_{ss},
\]
it follows that
\begin{align}
\Cov(Q,W)
&\approx
(v_{qq}-v_{qs})P,
\label{eq:closure-cov-QW}
\\
\Cov(S,W)
&\approx
(v_{qs}-v_{ss})P.
\label{eq:closure-cov-SW}
\end{align}
Using \(B=Q-W\), we further obtain
\begin{align}
\Cov(Q,B)
&\approx
v_{qq}(1-P)+v_{qs}P,
\label{eq:closure-cov-QB}
\\
\Cov(S,B)
&\approx
v_{qs}(1-P)+v_{ss}P.
\label{eq:closure-cov-SB}
\end{align}

For completeness, the corresponding raw mixed moments are
\begin{align}
\E[QW]
&\approx
m_Qw+(v_{qq}-v_{qs})P,
\label{eq:closure-EQW}
\\
\E[SW]
&\approx
m_Sw+(v_{qs}-v_{ss})P,
\label{eq:closure-ESW}
\\
\E[QB]
&\approx
m_Qb+v_{qq}(1-P)+v_{qs}P,
\label{eq:closure-EQB}
\\
\E[SB]
&\approx
m_Sb+v_{qs}(1-P)+v_{ss}P.
\label{eq:closure-ESB}
\end{align}
These expressions may be substituted directly into the raw moment equations
\eqref{eq:exact-raw-QQ}--\eqref{eq:exact-raw-QS}.

Substituting
\eqref{eq:gaussian-expected-waiting}--\eqref{eq:closure-cov-SB}
into the exact mean, variance, and covariance equations gives the closed
five-dimensional system
\begin{align}
\dot m_Q
={}&
\lambda-\mu b-\theta w,
\label{eq:closed-mean-Q}
\\
\dot m_S
={}&
\gamma(c-m_S)-p\mu b,
\label{eq:closed-mean-S}
\\
\dot v_{qq}
={}&
\lambda+\mu b+\theta w
-2\mu\left[v_{qq}(1-P)+v_{qs}P\right]
\nonumber\\
&-2\theta(v_{qq}-v_{qs})P,
\label{eq:closed-var-Q}
\\
\dot v_{ss}
={}&
\gamma(c-m_S)+p\mu b
-2\gamma v_{ss}
\nonumber\\
&-2p\mu\left[v_{qs}(1-P)+v_{ss}P\right],
\label{eq:closed-var-S}
\\
\dot v_{qs}
={}&
-\mu\left[v_{qs}(1-P)+v_{ss}P\right]
-\theta(v_{qs}-v_{ss})P
\nonumber\\
&-\gamma v_{qs}
-p\mu\left[v_{qq}(1-P)+v_{qs}P\right]
+p\mu b,
\label{eq:closed-cov-QS}
\end{align}
where
\begin{align}
\sigma_D
&=
\sqrt{v_{qq}+v_{ss}-2v_{qs}},
\label{eq:closed-sigma-D}
\\
a
&=
\frac{m_Q-m_S}{\sigma_D},
\label{eq:closed-a}
\\
P
&=
\Phi(a),
\label{eq:closed-P}
\\
w
&=
\sigma_D\phi(a)+(m_Q-m_S)P,
\label{eq:closed-w}
\\
b
&=
m_Q-w.
\label{eq:closed-b}
\end{align}

The stationary Gaussian-closure moments are obtained by setting the
left-hand sides of
\eqref{eq:closed-mean-Q}--\eqref{eq:closed-cov-QS} equal to zero and solving
the resulting nonlinear algebraic system. Alternatively, the transient
system can be integrated numerically until the moments converge.

If
\[
\sigma_D^2=v_{qq}+v_{ss}-2v_{qs}
\]
is numerically zero, the Gaussian expressions are interpreted through their
deterministic limits:
\[
w=(m_Q-m_S)^+,
\qquad
b=m_Q\wedge m_S,
\]
and
\[
P=
\begin{cases}
1, & m_Q>m_S,\\
0, & m_Q<m_S,\\
1/2, & m_Q=m_S.
\end{cases}
\]

\section{Second-Order Properties of Excess Demand and Idle Capacity}
\label{app:positive-part}
\begin{figure}[t]
\centering

\begin{minipage}{0.48\textwidth}
    \captionsetup{font=scriptsize}
    \centering
    \includegraphics[width=\linewidth]{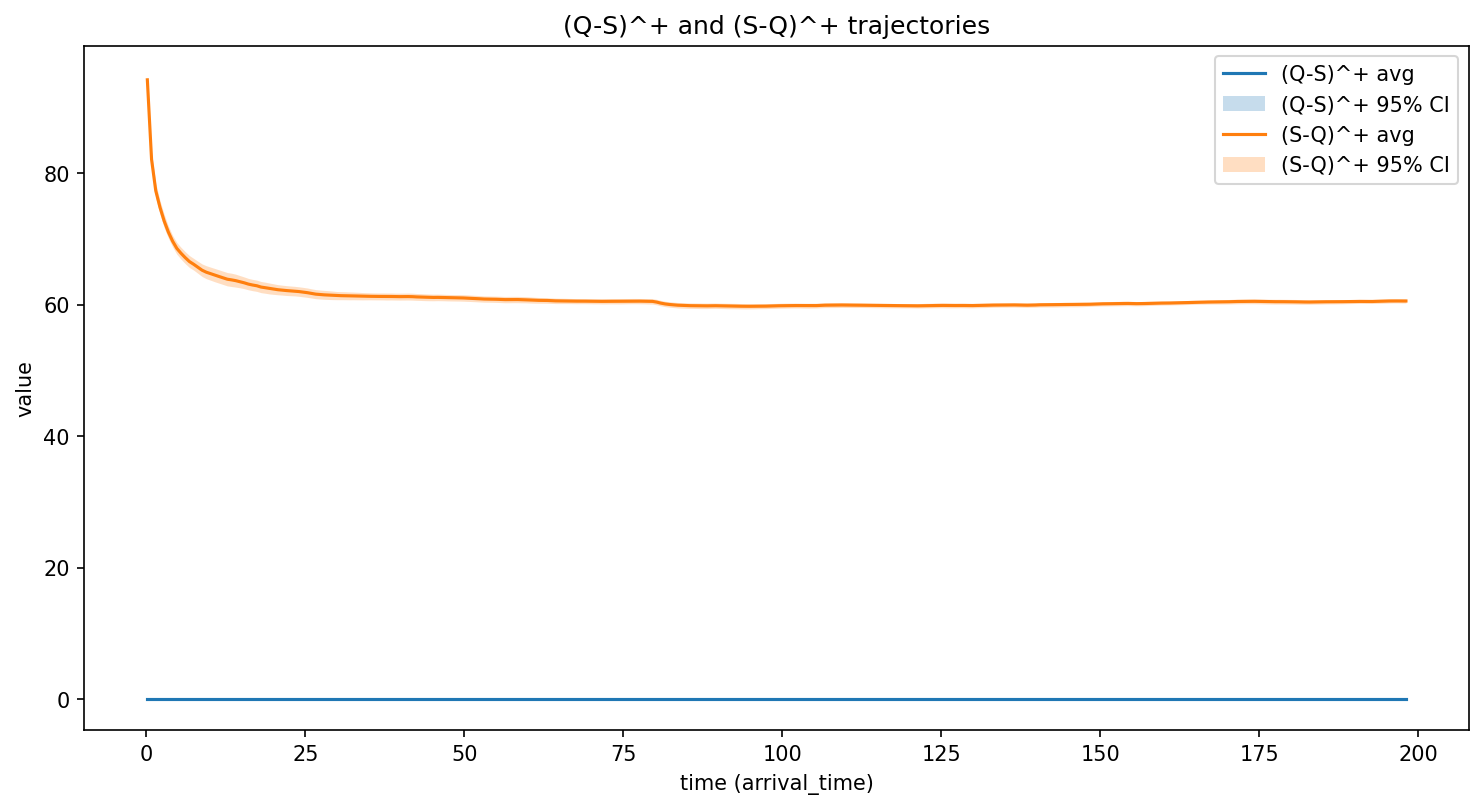}
    {\scriptsize \caption*{(a) $Q(t)$ and $S(t)$, $\lambda = 100,\, \mu = 5 , \, \theta = 1,\, p=0.1,\, \gamma = 0.5$ and $c=100$ (UL)}}
\end{minipage}
\hfill
\begin{minipage}{0.48\textwidth}
    \captionsetup{font=scriptsize}
    \centering
    \includegraphics[width=\linewidth]{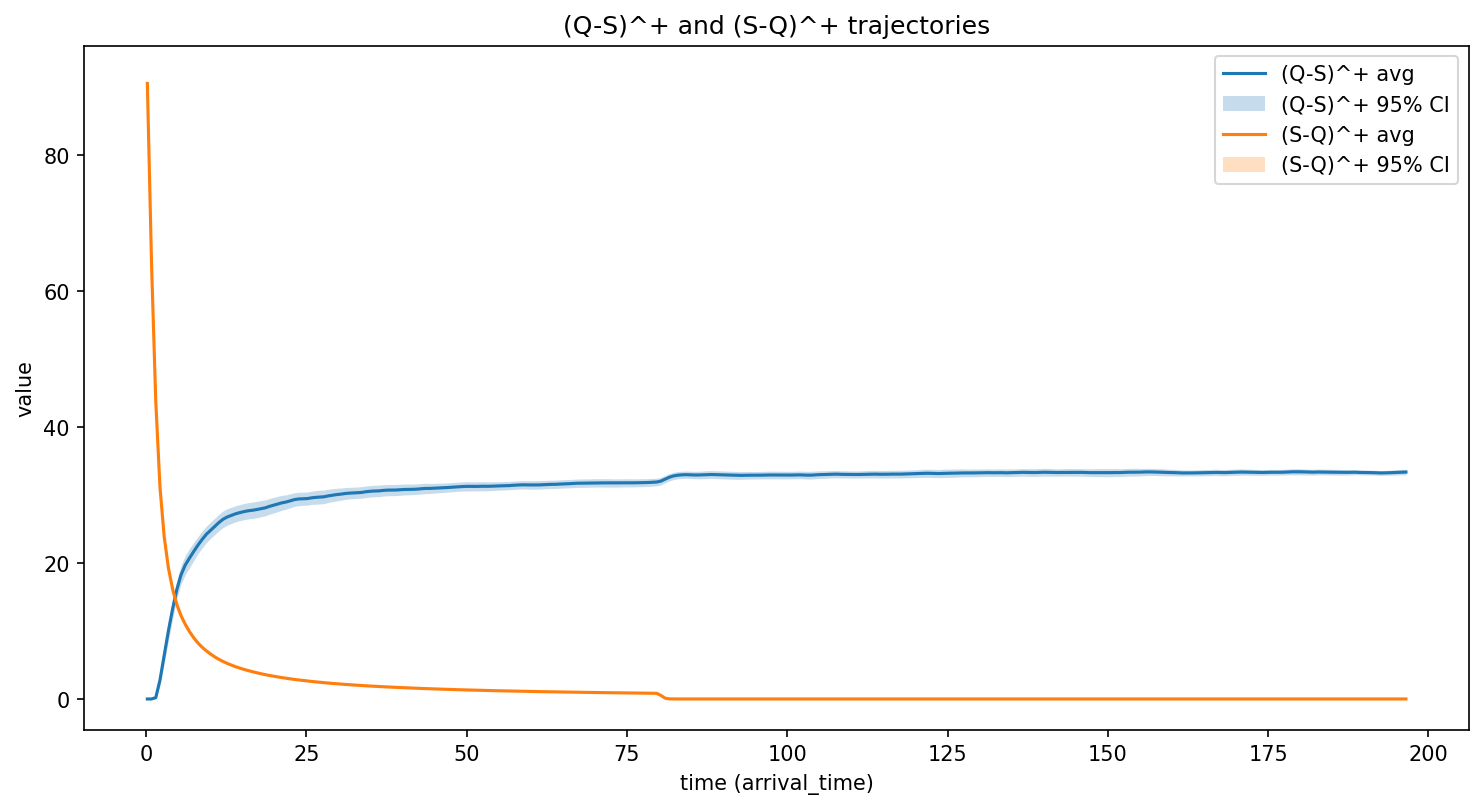}
    {\scriptsize\caption*{(b) $Q(t)$ and $S(t)$, $\lambda = 100,\, \mu = 1 , \, \theta = 1,\, p=0.5,\, \gamma = 1$ and $c=100$ (OL)}}
\end{minipage}

\vspace{0.5cm}

\begin{minipage}{0.48\textwidth}
    \captionsetup{font=scriptsize}
    \centering
    \includegraphics[width=\linewidth]{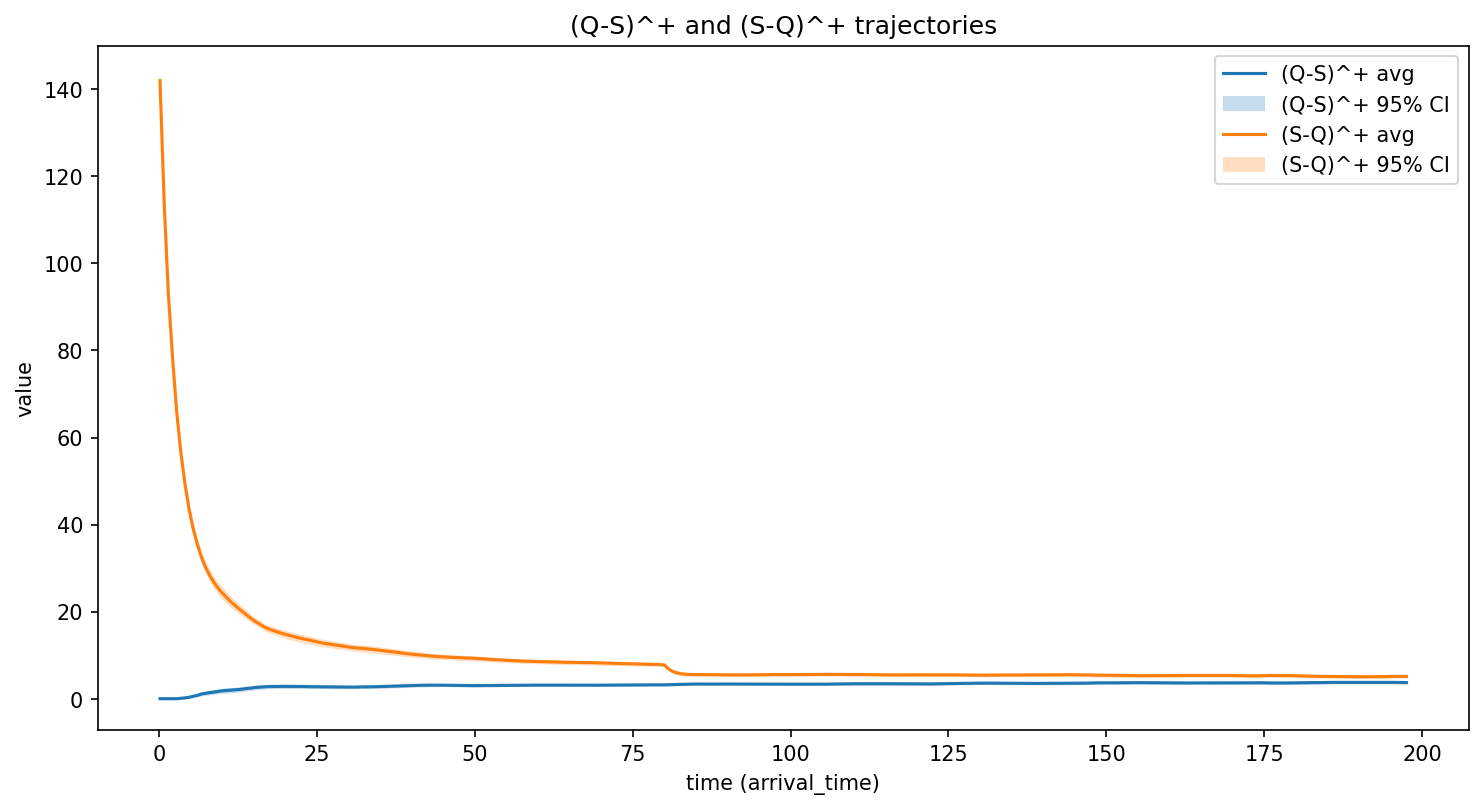}
    {\scriptsize\caption*{(c) $Q(t)$ and $S(t)$, $\lambda = 100,\, \mu = 1 , \, \theta = 1,\, p=0.5,\, \gamma = 1,\, c=150$ (near-critical UL)}}
\end{minipage}

\caption{$(Q(t)-S(t))^+$ and $(S(t)-Q(t))^+$ trajectories for 3 parameter sets and 100 simulation runs of $10,000$ customers.}
\label{fig:Q-S_pos-trajectories}
\end{figure}

\begin{figure}[t]
\centering

\begin{minipage}{0.48\textwidth}
    \captionsetup{font=scriptsize}
    \centering
    \includegraphics[width=\linewidth]{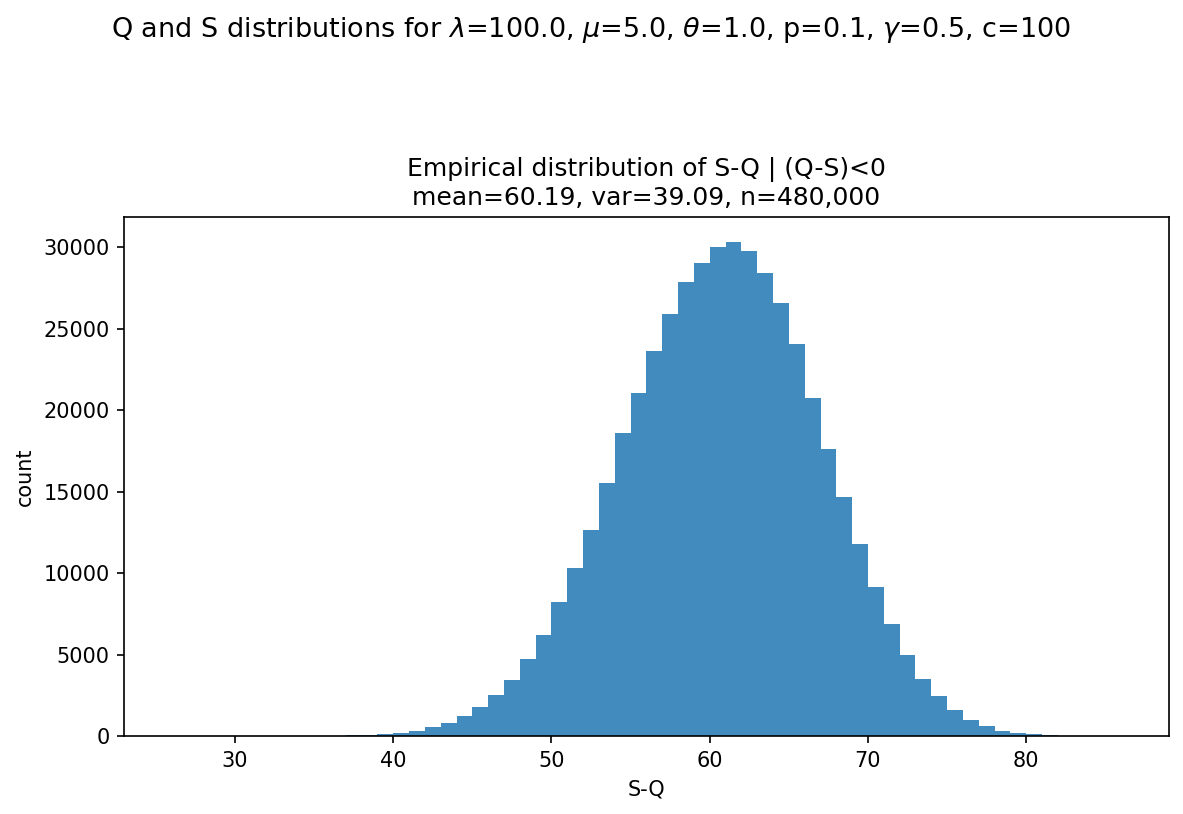}
    {\scriptsize \caption*{(a) Distribution of $(S(t)-Q(t))^+$, $\lambda = 100,\, \mu = 5 , \, \theta = 1,\, p=0.1,\, \gamma = 0.5$ and $c=100$ (UL)}}
\end{minipage}
\hfill
\begin{minipage}{0.48\textwidth}
    \captionsetup{font=scriptsize}
    \centering
    \includegraphics[width=\linewidth]{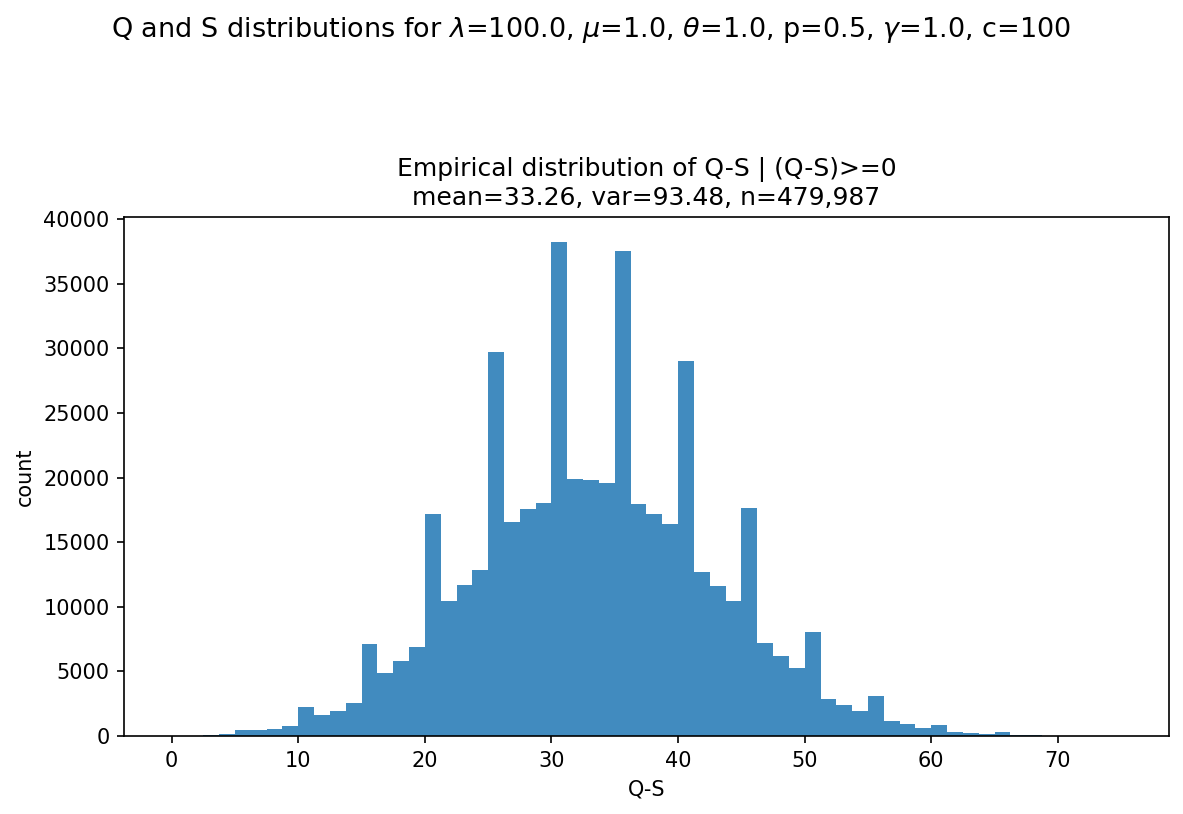}
    {\scriptsize\caption*{(b) Distribution of $(Q(t)-S(t))^+$, $\lambda = 100,\, \mu = 1 , \, \theta = 1,\, p=0.5,\, \gamma = 1$ and $c=100$ (OL)}}
\end{minipage}

\vspace{0.5cm}

\begin{minipage}{0.48\textwidth}
    \captionsetup{font=scriptsize}
    \centering
    \includegraphics[width=\linewidth]{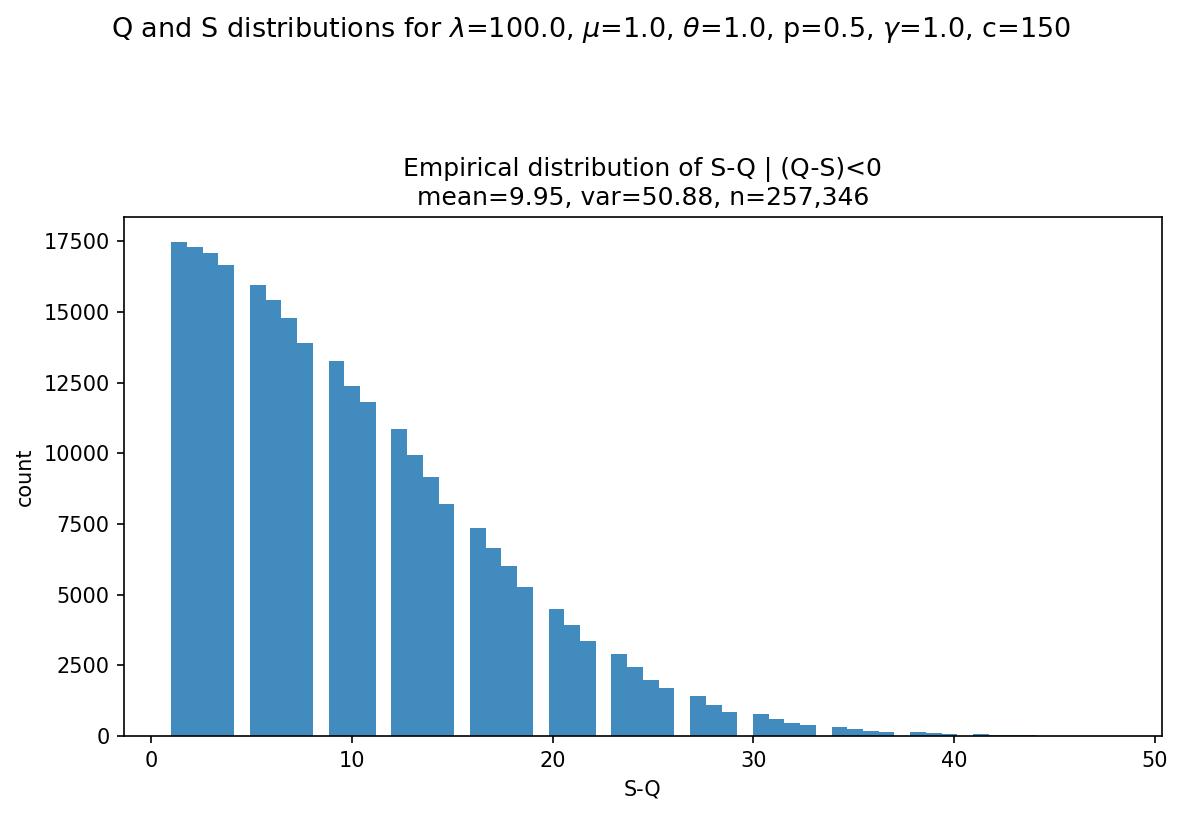}
    {\scriptsize\caption*{(c) Distribution of $(Q(t)-S(t))^+$, $\lambda = 100,\, \mu = 1 , \, \theta = 1,\, p=0.5,\, \gamma = 1,\, c=150$ (near-critical UL)}}
\end{minipage}

\caption{Gap process distribution}
\label{fig:Q-S_pos-distribution}
\end{figure}

In queueing systems with stochastic server availability, performance measures are often expressed in terms of excess demand or excess capacity.
In particular, the random variables $(Q-S)^+$ (similar to the gap process we looked at in Section~\ref{sec:diffusion}) and $(S-Q)^+$ represent, respectively, the instantaneous congestion beyond available capacity and the amount of idle service capacity. Moreover, we also know that the variance of $(Q-S)^+$ is key in studying the tail risk of abandonment, which is directly controlled by CVaR-type objectives.
While earlier sections focus on mean performance measures, understanding the variability of these quantities is important for assessing the reliability of staffing rules and the dispersion of delay and abandonment outcomes.
Under the bivariate-normal approximation for $(Q,S)$, both quantities reduce to the positive part of a Gaussian random variable, allowing their second-order properties to be characterized in closed form.

Notice that the moments of the positive part of a normal random variable are well known. Let \(X^{+} = \max\{X,0\}\), were \(X \sim \mathcal{N}(\mu,\sigma^{2})\), and define \(\alpha = \mu/\sigma\). Then
\begin{equation}\label{eq:positive-part-var}
    \begin{split}
        \mathbb{E}[X^{+}] &= \mu\,\Phi(\alpha) + \sigma\,\phi(\alpha),\\
        \mathbb{E}[(X^{+})^{2}] &= (\mu^{2}+\sigma^{2})\Phi(\alpha)
    + \mu\sigma\,\phi(\alpha),\\
    \mathrm{Var}(X^{+}) &= (\mu^{2}+\sigma^{2})\Phi(\alpha)
    + \mu\sigma\,\phi(\alpha)
    - \big(\mu\Phi(\alpha)+\sigma\phi(\alpha)\big)^{2}.
    \end{split}
\end{equation}
Here, \(\phi(\cdot)\) and \(\Phi(\cdot)\) denote the standard normal density and distribution functions, respectively. A derivation is provided in Appendix~\ref{app:normal_pos_sec_moment}. 

Applying \eqref{eq:positive-part-var} with \(X\sim\mathcal N(m,\sigma^2)\) yields
\[
\mathrm{Var}\!\bigl((Q-S)^{+}\bigr),
\]
while replacing \(m\) by \(-m\) yields
\[
\mathrm{Var}\!\bigl((S-Q)^{+}\bigr),
\]
where $m$ and $\sigma_D$ are as defined in \ref{sec:abandon}

\subsubsection{Proof of identities in \eqref{eq:positive-part-var}}\label{app:normal_pos_sec_moment}
\begin{proof}
Let $X\sim \mathcal N(\mu,\sigma^2)$ and define $X^+=\max\{X,0\}$. Write $X=\mu+\sigma Z$ with
$Z\sim \mathcal N(0,1)$, and set $\alpha=\mu/\sigma$. Then
\[
X^+ = (\mu+\sigma Z)\,\mathbf 1\{ \mu+\sigma Z>0\} = (\mu+\sigma Z)\,\mathbf 1\{Z>-\alpha\}.
\]

\medskip
Using linearity,
\[
\E[X^+] = \mu\,\Pr(Z>-\alpha) + \sigma\,\E\big[Z\,\mathbf 1\{Z>-\alpha\}\big].
\]
Since $\Pr(Z>-\alpha)=\Phi(\alpha)$, it remains to compute $\E[Z\mathbf 1\{Z>a\}]$ for $a=-\alpha$.
For $Z\sim \mathcal N(0,1)$ with density $\phi$, integration by parts gives
\[
\E\big[Z\,\mathbf 1\{Z>a\}\big]
=\int_a^\infty z\,\phi(z)\,dz
=\Big[-\phi(z)\Big]_{z=a}^{\infty}
=\phi(a),
\]
because $\phi(z)\to 0$ as $z\to\infty$ and $\phi'(z)=-z\phi(z)$.
With $a=-\alpha$ and $\phi(-\alpha)=\phi(\alpha)$, we obtain
\[
\E[X^+] = \mu\,\Phi(\alpha)+\sigma\,\phi(\alpha).
\]
\medskip
Similarly as how we computed the previous expectation, we have that,
\begin{align*}
    \E[(X^+)^2] &=\E\big[(\mu+\sigma Z)^2\,\mathbf 1\{Z>-\alpha\}\big] \\
    &=\mu^2\,\Pr(Z>-\alpha) + 2\mu\sigma\,\E[Z\mathbf 1\{Z>-\alpha\}]
    +\sigma^2\,\E[Z^2\mathbf 1\{Z>-\alpha\}].
\end{align*}
We already have $\Pr(Z>-\alpha)=\Phi(\alpha)$ and $\E[Z\mathbf 1\{Z>-\alpha\}]=\phi(\alpha)$.
It remains to compute $\E[Z^2\mathbf 1\{Z>a\}]$ with $a=-\alpha$:
\[
\int_a^\infty z^2\phi(z)\,dz
=\int_a^\infty \big(\phi(z) - (z\phi(z))'\big)\,dz
=\int_a^\infty \phi(z)\,dz - \Big[z\phi(z)\Big]_{z=a}^{\infty}
=\bigl(1-\Phi(a)\bigr) + a\phi(a).
\]
(Here we used $(z\phi(z))'=\phi(z)-z^2\phi(z)$ and again $z\phi(z)\to 0$ as $z\to\infty$.)
Substituting $a=-\alpha$ yields
\[
\E[Z^2\mathbf 1\{Z>-\alpha\}]
= 1-\Phi(-\alpha) + (-\alpha)\phi(-\alpha)
= \Phi(\alpha) - \alpha\phi(\alpha).
\]
Therefore
\begin{align*}
\E[(X^+)^2]
&=\mu^2\,\Phi(\alpha) + 2\mu\sigma\,\phi(\alpha)
+\sigma^2\bigl(\Phi(\alpha)-\alpha\phi(\alpha)\bigr) \\
&=(\mu^2+\sigma^2)\Phi(\alpha) + \mu\sigma\,\phi(\alpha),
\end{align*}
since $\alpha=\mu/\sigma$ implies $\sigma^2\alpha\phi(\alpha)=\mu\sigma\phi(\alpha)$.

Finally,
\[
\Var(X^+) = \E[(X^+)^2]-\big(\E[X^+]\big)^2
=(\mu^{2}+\sigma^{2})\Phi(\alpha)
+\mu\sigma\,\phi(\alpha)
-\big(\mu\Phi(\alpha)+\sigma\phi(\alpha)\big)^{2}.
\]
This completes the proof.
\end{proof}

\subsection{Empirical distributions of excess demand and idle capacity}
\label{app:positive-part-distributions}

The positive-part processes
\[
W(t):=(Q(t)-S(t))^+,
\qquad
I(t):=(S(t)-Q(t))^+,
\]
measure excess demand and idle capacity, respectively. Their distributions
complement the mean-based delay and abandonment metrics studied in the main
text. In particular, the upper tail of $W(t)$ is relevant for reliability and
risk-sensitive congestion measures, whereas the distribution of $I(t)$
quantifies variability in unused capacity. Under the joint-normal
approximation, their moments and tail probabilities can be computed from the
distribution of $Q-S$ and may be used to construct quantile- or CVaR-based
staffing objectives. We do not pursue such objectives here.

Figures~\ref{fig:Q-S_pos-trajectories}
and~\ref{fig:Q-S_pos-distribution} illustrate how these quantities change
across staffing regimes. In the underloaded regime, $W(t)$ is frequently zero
and most fluctuations appear in idle capacity; in the overloaded regime,
$I(t)$ is frequently zero and excess demand dominates. Near the critical
capacity, the system switches between the two states, producing point masses
at zero and right-skewed positive values for both quantities. These
distributions illustrate why mean performance measures alone may not fully
describe congestion and capacity risk near the boundary.

\end{document}